\documentclass[a4paper,leqno,11pt]{amsart}\usepackage{hyperref}

\usepackage{palatino}
\usepackage{csquotes}
\usepackage{amsfonts}
\usepackage{graphicx}
\usepackage{amsmath,latexsym,amssymb,amsthm}

\usepackage{graphicx}
\usepackage{amsmath}
\usepackage{epstopdf}
\usepackage{mathrsfs}
\usepackage{epsfig}
\usepackage{hyperref}
\usepackage{ifthen}
\usepackage{float}
\usepackage{enumerate}
\usepackage{mathtools}
\usepackage[bottom]{footmisc}
\usepackage{tikz}
\usetikzlibrary{matrix,shapes,arrows,positioning,chains}
\usepackage{amssymb,amsmath,amsfonts}

\newtheorem{thm}{Theorem}[section]
\newtheorem{prop}[thm]{Proposition}
\newtheorem{lem}[thm]{Lemma}
\newtheorem{defn}[thm]{Definition}

\newtheorem{rmk}[thm]{Remark}
\newtheorem{cor}[thm]{Corollary}
\usepackage{mathrsfs}
\usepackage{tikz}
\numberwithin{thm}{section}

\DeclareMathOperator\supp{supp}

\newcommand{\R}{\mathbb{R}}
\newcommand{\CZO}{Calder\'on-Zygmund operators}
\newcommand{\CZ}{Calder\'on-Zygmund}
\newcommand{\MCZ}{multilinear Calder\'on-Zygmund}

\author[A. Ghosh,  A. Bhojak, P. Mohanty and S. Shrivastava]{Abhishek Ghosh,  Ankit Bhojak, Parasar Mohanty and Saurabh Shrivastava}
\address{
Abhishek Ghosh\\
Department of Mathematics and Statistics\\
Indian Institute of Technology Kanpur\\
Kanpur-208016, India.}
\email{abhig@iitk.ac.in}
\address{
Ankit Bhojak\\
Department of Mathematics and Statistics\\
Indian Institute of Technology Kanpur\\
Kanpur-208016, India.}
\email{ankitbjk@iitk.ac.in}
\address{
Parasar Mohanty\\
Department of Mathematics and Statistics\\
Indian Institute of Technology Kanpur\\
Kanpur-208016, India.}
\email{parasar@iitk.ac.in}
\address{
Saurabh Shrivastava\\
Department of Mathematics\\
Indian Institute of Science Education and Research Bhopal, Bhopal-462066, India.}
\email{saurabhk@iiserb.ac.in}


\begin{document}

\title[Multilinear Calder\'{o}n-Zygmund operators on non-homogeneous spaces]{Sharp weighted estimates for multilinear Calder\'{o}n-Zygmund operators on non-homogeneous spaces}


%
%

\begin{abstract}
In this article, we address pointwise sparse domination for multilinear Calder\'on-Zygmund operators on upper doubling, geometrically doubling metric measure spaces. As a consequence, we have obtained sharp quantitative weighted estimates for multilinear Calder\'on-Zygmund operators. 
\end{abstract}
\maketitle
\section{Introduction}
Pointwise domination of Calder\'on-Zygmund operators by some special discrete operators (called as ``sparse operators") was initiated by the works of Conde-Alonso and Rey in \cite{CR}, Lerner and Nazarov in \cite{LN}. In \cite{Lacey}, Lacey developed an innovative approach to dominate \CZ~operators by sparse operators and relaxed the log-Dini condition to standard Dini condition. Subsequently, Lerner obtained an elegant proof of the pointwise domination of \CZ~operators by sparse operators in \cite{L15}. In \cite{GT}, Grafakos and Torres initiated the study of multilinear \CZO \; on $\R^n$. Later in \cite{GT1}, they have shown that the multilinear Calder\'on-Zygmund operator can be controlled by the product of Hardy-Littlewood maximal functions in terms of norms and the authors have concluded some weighted boundedness of the multilinear Calder\'on-Zygmund operators in terms of linear Muckenhoupt classes. The question of controlling a multilinear Calder\'on-Zygmund operator by a smaller maximal operator and developing the multilinear weighted theory is addressed by Lerner \textit{et al} in \cite{LOP}. Later its sharp quantitative dependence on $A_{\vec{P}}$ characteristic was addressed in \cite{LMS} by adopting recently developed sparse domination techniques in the multilinear settings. In \cite{DLi}, pointwise sparse domination for multilinear \CZ~operators was achieved in Euclidean setting. Our goal in this article is to address the issue of dominating multilinear Calder\'on-Zygmund operators by sparse operators in the non-homogeneous setting and to achieve quantitative weighted estimates. Our domination of the multilinear Calder\'on-Zygmund operator will be pointwise. To state our results in details we recall the following preliminaries. 

In \cite{T1,T2},  X. Tolsa established the existence of the principal value of the Cauchy integral operator on $\mathbb C$ with respect to arbitrary positive Radon measures $\mu$ satisfying $\mu(D_r)\leq C r$ where $D_r$ represents a disc with radius $r$ and $C$ is a constant independent of $r$.  These works of Tolsa played a deep role in proving the celebrated Painlev\'e's problem (\cite{T2,T3,T4}) as well as a prelude to the study of \CZO \; on non-homogeneous spaces.  Nazarov, Treil, and Volberg (\cite{NTV2,NTV1,NTV3}) have developed an innovative approach to treat \CZO\;  on non-homogeneous spaces.  They obtained the weak end-point result from the estimates on point masses and Whitney decomposition. Tolsa in \cite{T3,T4} developed a suitable \CZ\; decomposition on $\R^n$  with non-doubling measures of polynomial growth, as a consequence, he obtained the weak type end-point boundedness of \CZ~operators on non-homogeneous spaces (see \cite{Tweak}).  T. Hyt\"onen, in \cite{Hyt10}, provided a unified framework to study \CZO\;  in upper doubling, geometric doubling metric measure spaces. Throughout this article  we refer an upper doubling, geometric doubling metric measure space as ``non-homogeneous" space. We start by defining upper doubling measures.
\begin{defn}[\cite{Hyt10,vk}]
A  metric measure space $(X,d,\mu)$ is {\it{upper doubling}} if there exists a dominating function $\lambda:X\times(0,\infty)\rightarrow(0,\infty)$ and a constant $C_{\lambda} > 0$ such that
\begin{enumerate}[i)]
\item
$\mu(B(x,r))\leq \lambda(x,r)$ for all $x\in X, r>0,$ where $B(x,r)$ denotes the ball of radius $r$ with center at $x\in X$.
\item
For every $x\in X$ the function $r\rightarrow \lambda(x,r)$ is non-decreasing.
\item
$\lambda$ satisfies the doubling property: $\lambda(x,r) \leq C_{\lambda} \lambda(x,\frac{r}{2})$ for all $x\in X, r>0$.
\end{enumerate}
 In particular, if $\lambda(x,r) =C r^d$ for some $d>0$ and constant $C>0$  then we say $\mu$ has polynomial growth.

\label{def1}
\end{defn}
  
   \begin{defn}[\cite{vk}]
A metric space $(X,d)$ is called geometric doubling(with doubling dimension $N$) if any ball $B(x,r)\subset X$ can contain at most $N$ centres $\{y_i\}$ of disjoint balls $\{B(y_i,r/4)\}_{i=1}^N$.
\label{gdmeas}
\end{defn}
Recently in \cite{vk}, A. Volberg and P. Zorin-Kranich used the sparse domination to prove the appropriate weighted boundedness of the  Calder\'on-Zygmund operators on non-homogeneous spaces. In this paper, we will address the pointwise domination of multilinear \CZO~by appropriate sparse operators on non-homogeneous spaces.  The weak type weighted boundedness of multilinear Calder\'on-Zygmund operators on non-homogeneous spaces has been considered in \cite{Hu}. In this work, they have assumed the unweighted weak type end-point boundedness for multilinear \CZO.

 The goal of this article is twofold, first, we prove the natural end-point boundedness of \MCZ~operators on non-homogeneous spaces. Secondly, we provide appropriate sparse domination results and quantitative weighted estimates. Using sparse domination techniques we have obtained strong (in contrast to weak type bound in \cite{Hu}) weighted boundedness for multilinear \CZO.  Besides, as we dominate \MCZ~operators pointwise by sparse operators,  our result addresses more exponents than given in \cite{LMS}.  Our bound is also sharp (when restricted to Euclidean setting) and holds with a much weaker assumption than \cite{Hu}. Also, we have proved an analogue of  Cotlar's type inequality for multilinear maximal \CZO\; in this setting. The end-point boundedness of \MCZ~operators on $\R^n$ for polynomial growth measures was addressed in \cite{marti}. Throughout this article $\vec{P}$ will denote the tuple $\vec{P}=(p_{1},p_{2})$ with $1\leq p_{1},p_{2}< \infty$ and $\frac{1}{p}=\frac{1}{p_{1}}+\frac{1}{p_{2}}$.  Given a tuple of weights $\vec{w}=(w_{1},w_{2})$ and a tuple $\vec{P}=(p_1,p_2)$ we set $v_{\vec{w}}=\prod_{j=1}^{2} w_{j}^{\frac{p}{p_j}}$. $A\lesssim B$ is the abbreviation for $A\leq CB$, where $C$ is independent of $A$ and $B$. 
\section{Definitions and Results}
We start by defining multilinear \CZO~ on non-homogeneous spaces.
\begin{defn}
\label{c6defMCZ}
Let $(X,d,\mu)$ be a non-homogeneous space. A Calder\'on-Zygmund kernel $K$ is $\mu$-locally integrable function defined away from the diagonal $x=y_1=y_2$ in $X^3$ and satisfying the following:\\
 \textbf{Size condition:} There exist a constant $C_{K}>0$ such that whenever $x\neq y_j$ for some $j=1,2$,
\begin{align}
\label{size}|K(x,y_1,y_2)|\leq C_{K} \frac{1}{(\lambda(x,d(x,y_1))+\lambda(x,d(x,y_2)))^2}.
\end{align}
\textbf{Regularity conditions:} 
There exist a constant $C_{K}>0$ such that the following regularity conditions hold. 
\begin{align}
\label{reg1}
|K(x,y_1,y_2)-K(x',y_1,y_2)|
&\leq C_{K}\frac{1}{(\lambda(x,d(x,y_1))+\lambda(x,d(x,y_2)))^2}~\omega\textstyle{\left(\frac{d(x,x')}{\sum\limits_{i=1}^2 d(x,y_i)}\right)},
\end{align}
whenever $d(x,x')\leq \frac{1}{2} \max\limits_{i=1,2} d(x,y_i)$. Also similar regularity conditions hold for other components. $\omega$ is a modulus of continuity satisfying the $Dini(1)$ condition i.e., $\omega$ is a monotonically increasing subadditive function with $w(0)=0$ and $\int\limits_{0}^{1}\frac{\omega(t)}{t} dt=||\omega||_{Dini(1)} < \infty$.
\end{defn}
\begin{defn} We say that an operator $T$ is a bilinear Calder\'on-Zygmund operator if it is bounded from $L^{p_1}(X, d, \mu)\times L^{p_2}(X,d,\mu)$ into $L^p(X, d, \mu)$ where $\frac{1}{p}=\sum\limits_{i=1}^{2}{\frac{1}{p_i}}, 1<p,p_1,p_2<\infty$ and there exists a Calder\'on-Zygmund kernel $K$ such that $T$ is represented  for compactly supported bounded functions as,
\begin{equation}
T(f_1,f_2)(x)=\int_{X^2}K(x,y_1,y_2) f_1(y_1)f_2(y_2)d\mu(y_1)d\mu(y_2)
\label{cz}
\end{equation}
 for a.e $x \notin \bigcap\limits_{j=1}^{2}\supp(f_j)$. Corresponding to a bilinear Calder\'on-Zygmund operator $T$, we define the bilinear maximal truncated operator by
\begin{equation}
T^{*}(f_1,f_2)(x)=\sup\limits_{\epsilon>0}|T_{\epsilon}(f_1,f_2)(x)|,
\label{maxCZ}
\end{equation}
where $T_\epsilon$ is defined as following,
$$T_{\epsilon}(f_1,f_2)(x)=\int\limits_{\sum\limits_{i=1}^2 d(x,y_i)^2>\epsilon^2}K(x,y_1,y_2)f_1(y_1)f_2(y_2)d\mu(y_1)d\mu(y_2).$$
\end{defn}
We would like to mention that our definition does not demand the weak-type end point estimate like \cite{Hu}. We require another form of the truncated maximal function for geometric compatibility and it is easier to work with the following truncated maximal operator (see~\cite{GT1} for more details).
\begin{align*}
\tilde{T}^{*}(f_1,f_2)(x)=\sup\limits_{\epsilon>0}|\tilde{T}_{\epsilon}(f_1,f_2)(x)|,
\end{align*}
where, $$\tilde{T}_{\epsilon}(f_1,f_2)(x):=\int\limits_{\max\limits_{i=1,2} d(x,y_i)>\epsilon}K(x,y_1,y_2)f_{1}(y_1)f_{2}(y_2)d\mu(\vec{y}).$$
The following remark says the truncated maximal operators are equivalent upto a maximal function.
\begin{rmk}
\label{comparison}
Denote $S_\epsilon(x)=\lbrace\vec{y}\in X^{2}:\max\limits_{i=1,2} d(x,y_i)\leq \epsilon\rbrace$ and $U_\epsilon(x):=\{\vec{y}\in S_\epsilon(x):d(x,y_1)^2+d(x,y_2)^2\geq \epsilon^2\}$. It is very easy to observe
\begin{align*}
|T_{\epsilon}(f_1,f_2)(x)-\tilde{T_{\epsilon}}(f_1,f_2)(x)|&\leq \int_{U_\epsilon(x)}|K(x,y_1,y_2)||f_{1}(y_1)||f_{2}(y_2)|d\mu(y_1)d\mu(y_2)\\
&\leq C_{K,\lambda}\mathcal{M}_{\lambda}(f_1,f_2)(x),
\end{align*}
where, 
\begin{align}
\label{revmax}
\mathcal{M}_{\lambda}(f_1,f_2)(x):=\sup_{r>0} \frac{1}{\lambda(x,r)^2}\int_{B(x,r)} |f_{1}(y_{1})|  d\mu( y_{1}) \int_{B(x,r)} |f_{2}(y_{2})|  d\mu( y_{2}).
\end{align}
Therefore, $T^{*}(f_1,f_2)$ and $\tilde{T}^{*}(f_1,f_2)$ are pointwise a.e. equivalent upto the maximal function $\mathcal{M}_{\lambda}$. The maximal operator $\mathcal{M}_{\lambda}$ plays a crucial role similar to that of Hardy-Littlewood maximal function in \cite{L15}. It is easy to verify that $\mathcal{M}_{\lambda}$ is bounded from $L^{1}(X,d,\mu)\times L^1(X,d,\mu)$ to $L^{\frac{1}{2},\infty}(X,d,\mu)$ as $\mathcal{M}_{\lambda}$ is bounded by $M^c_{\mu}(f_1)(x)M^c_{\mu}(f_2)(x)$ pointwise and $M^c_{\mu}$ is weak-type (1,1) with respect to $\mu$, where $M^c_\mu$ is the centred Hardy-Littlewood maximal function with respect to $\mu$.
\end{rmk}
Now we state the main sparse domination principle of this article. Our strategy of sparse domination for multilinear Calder\'on-Zygmund operator is motivated by the work of Volberg and Kranich in \cite{vk}. In the absence of a compatible dyadic structure the authors in \cite{vk} have used the David-Mattila cells. We also need them for our purpose and a detailed description of David-Mattila cells is given in Lemma~\ref{dmlemma} (see also \cite{dm,vk}). $\mathscr{D}$ will denote the collection of David-Mattila cells and corresponding to each cell $Q\in \mathscr{D}$, $B(Q)$ will denote the associated ball with respect to $Q$ satisfying certain properties. We recall the definition of a sparse family in this setting. 
\begin{defn}
A family of David-Mattila cells $\mathcal{S}\subset \mathscr{D}$ is said to be $\eta$-sparse, $0<\eta<1$, if for every $Q\in \mathcal{S}$ there exists a measurable set $E_{Q}\subset Q$ such that
\begin{enumerate}[i)]
\item $\mu(E_{Q})\geq \eta
\mu(Q)$ and
\item the sets $\lbrace E_{Q}\rbrace_{Q\in \mathcal{S}}$ are pairwise disjoint.
\end{enumerate}
Given a sparse family $\mathcal S$ in $X$ and a large number $\alpha\geq 200$, the bilinear sparse operator is defined by
$$\mathcal{A_S}(\vec{f})(x)=\sum\limits_{Q\in \mathcal{S}}\left( \prod_{i=1}^{2}\frac{1}{\mu(\alpha B(Q))}\int_{30B(Q)}|f_i| d\mu\right)\chi_{Q}. $$
\end{defn}
Now we are in a position to state the main sparse domination principle.
\begin{thm}
\label{sparsethm} 
Let $T$ be a bilinear Calder\'{o}n-Zygmund operator defined on a non-homogeneous space $(X,d,\mu)$. Let $\alpha\geq 200$ . Then for every bounded subset $X'$ of $X$ and for integrable functions $f_j$, with $\supp(f_j)\subseteq X'$ for $j=1,2$, there exist sparse families $\mathcal G_n,~n\geq 0$, of David-Mattila cells such that the sparse domination 
\begin{eqnarray}\label{sparsedom}
T^*(\vec{f})(x) \lesssim_{T, \alpha} \sum\limits_{n=0}^{\infty} 100^{-n}  \mathcal A_{{\mathcal G}_n}(\vec{f})(x)
\end{eqnarray}
holds pointwise a.e. on $X'$.
\end{thm}

In Proposition~\ref{sparseprop}, we will provide quantitative weighted boundedness of bilinear sparse operators. As a consequence of Theorem~\ref{sparsethm} and Proposition~\ref{sparseprop}, we obtain the following corollary concluding the weighted boundedness of bilinear \CZO~on non-homogeneous spaces.
\begin{cor} 
\label{c6mainresult} 
Let $T$ be a bilinear Calder\'{o}n-Zygmund operator defined on a non-homogeneous space $(X,d,\mu)$. Then for all exponents $1<p_1, p_2<\infty$ satisfying $\frac{1}{p}=\sum\limits_{i=1}^{2}\frac{1}{p_i}$ and any multiple weight $\vec{w}$, we have $$\|T^*(\vec{f})\|_{L^p(v_{\vec{w}})}\lesssim C_{w}\prod_{i=1}^{2}\|f_i\|_{L^{p_i}(w_i)},$$
where the implicit constant is independent of $\vec{w}$,  and 
\[   
C_{w} = 
     \begin{cases}
       \sup\limits_{Q}\frac{v_{\vec{w}}(Q)^{\frac{p'_0}{p}}\prod\limits_{i=1}^{2}\sigma_{i}(\alpha B(Q))^{\frac{p'_{0}}{p'_{i}}}}{\mu(\alpha B(Q))^{2}\mu(Q)^{2(p'_{0}-1)}}\;\; \text{if} ~\frac{1}{2}< p \leq 1\;\; and \;\; p_0=\min\limits_{i} p_i\\
       \sup\limits_{Q}\frac{v_{\vec{w}}(Q)\prod\limits_{i=1}^{2}\sigma_{i}(\alpha B(Q))^{\frac{p}{p'_i}}}{\mu(\alpha B(Q))^2 \mu(Q)^{2(p-1)}}\;\; \text{if} \;\; p\geq \max\limits_{i}p'_{i} \\
       
       \sup\limits_{Q} \frac{v_{\vec{w}}(Q)^{\frac{p'_{j}}{p}}\sigma_{j}(\alpha B(Q))\prod\limits_{i=1,i\neq j}^{2}\sigma_{i}(\alpha B(Q))^{\frac{p'_{j}}{p'_{i}}}}{\mu(\alpha B(Q))^2 \mu(Q)^{2(p'_{j}-1)}}\;\; \text{if} \;\; p'_{j}= \max\lbrace p,p'_{1},p'_{2}\rbrace.\\
     \end{cases}
\]
\end{cor}
\begin{rmk}
Let the underlying measure be a doubling measure and $\vec{w}\in A_{\vec{P}}(\mu)$ i.e., 
$$[\vec{w}]_{A_{\vec{P}}}:=\sup_{B} \left(\frac{1}{\mu(B)}\int_{B}v_{\vec{w}}  d\mu\right) \prod_{j=1}^{2}\left(\frac{1}{\mu(B)}\int_{B} w_{j}^{1-p'_{j}} d\mu\right) ^{\frac{p}{{p'_{j}}}}  \leq K<\infty.
$$
As in this case $\mu(\alpha B(Q))$ and $\mu(B(Q))$ are comparable thus in each of the cases mentioned above, we get  $C_{w}\simeq [\vec{w}]_{A_{\vec{P}}}^{\max(1,\frac{p'_1}{p},\frac{p'_2}{p})}$. Therefore in view of \cite{LMS}, we retrieve sharp constants for Euclidean setting. 

We also like to mention that in \cite{LMS}, the sparse domination was obtained in terms of norms as it was based on Lerner's mean oscillation formula. Therefore the authors  only get the sparse domination and sharp bound  for the range $1\leq p<\infty$, but as our method depends on pointwise domination, our result is true for the range $\frac{1}{2}< p <\infty$.
\end{rmk}
Our strategy for the proof of Theorem~\ref{sparsethm} depends on two ingredients. Firstly, we need to define an appropriate ``grand maximal truncated operator" and secondly, it is quite natural from the ideas developed in \cite{L15} and \cite{vk} that we need some sort of end-point boundedness of the ``grand maximal truncated operator" to conclude our result and to address this we have to first address the end-point boundedness of the bilinear Calder\'on-Zygmund operator in this setting. The following end-point estimate is an extension of the result (page 51 ,\cite{marti})   from polynomial growth measure to upper doubling,  geometric doubling metric measure.
 \begin{thm}
Let $T$ be  a bilinear Calder\'on-Zygmund operator, i.e the associated kernel $K$ satisfies all the properties in Definition~\ref{c6defMCZ} and $T$ is bounded from $L^{p_1}(X,d,\mu)\times L^{p_2}(X,d,\mu)$ to $L^{p}(X,d,\mu)$ where $\frac{1}{p}=\frac{1}{p_1}+\frac{1}{p_2}$ with $1<p,p_1,p_2<\infty$. Then $T$ is bounded from $L^1(X,d,\mu)\times L^1(X,d,\mu)$ to $L^{\frac{1}{2},\infty}(X,d,\mu)$.
\label{endpoint}
\end{thm}
 We also need the following version of the Cotlar's inequality in order to conclude the end-point boundedness of the maximally truncated bilinear \CZO. To state the result we need to define following operators.
   \begin{equation}
\mathscr{M}(f)(x):=\sup_{r>0}\frac{1}{\mu(B(x,5r))}\int_{B(x,r)}|f(y)|d\mu(y)
\end{equation}
\begin{equation}
 M_{\lambda}(f)(x):=\sup_{r>0}\frac{1}{\lambda(x,r)}\int_{B(x,r)}|f(y)| d\mu(y)
 \end{equation}
 From Vitali covering lemma it is straightforward to see both $\mathscr{M}$ and $M_{\lambda}$ are bounded from $L^{1}(X,d,\mu)$ to $L^{1,\infty}(X,d,\mu)$.
       \begin{thm}
\label{T*}
Let $T$ be a bilinear Calder\'{o}n-Zygmund operator defined on an  upper doubling, geometrically doubling metric measure space $(X,d,\mu)$. Then $T^*$ is bounded from $L^{1}(X,d,\mu)\times L^1(X,d,\mu)$ to $L^{\frac{1}{2},\infty}(X,d,\mu)$, in particular $T^*$ satisfies the Cotlar's type inequality:
\begin{align}
\label{cotlar}
\nonumber T^{*}(f_1,f_2)(x) &\leq C_{\eta,\lambda,T}M_{\lambda}(f_1)(x)M_{\lambda}(f_2)(x)+\mathscr{M}_{\eta}(|T(f_1,f_2)|)(x)\\&+\mathscr{M}(f_1)(x)\mathscr{M}(f_2)(x),
\end{align}
where $\eta\in (0,\frac{1}{2})$ and $\mathscr{M}_{\eta}(f):=(\mathscr{M}(|f|^\eta))^{\frac{1}{\eta}}.$
\end{thm}
\textbf{Presentation of the paper:} The proof of Theorem~\ref{endpoint} and Theorem~\ref{T*} is quite independent of the sparse domination principle, therefore to maintain the flow of the paper we postpone them to Section~\ref{endpointresults}. In Section~\ref{mainsection}, we provide proof of Theorem~\ref{sparsethm} and quantitative weighted estimates are achieved in Subsection~\ref{weighted estimates}.
\section{Sparse domination}
\label{mainsection}
The proof of the theorem is quite involved and it uses the bilinear grand maximal truncated operator and a recursive decomposition of David-Mattila cells in order to get the sparse domination. The notion of dyadic lattice in $\mathbb{R}^n$ with a non-doubling measure $\mu$ was introduced by David and Mattila~in \cite{dm}. In~\cite{vk}, the authors observed that the same construction of David-Mattila cells works in general in the case of a geometrically doubling metric measure space.
\begin{lem}[\cite{dm,vk}]
\label{dmlemma}
Let $(X,d,\mu)$ be a geometrically doubling metric measure space with doubling dimension $N$ and locally finite Borel measure $\mu$. If $W$ denotes the support of $\mu$ and $C_0>1$ and $A_0>5000 C_0$ are two given constants, then for each integer $k$, there exists a partition of $W$ into Borel sets $\mathcal D_k=\{Q\}_{Q\in \mathcal D_k}$  with the following properties
\begin{enumerate}[i)]
\setlength\itemsep{1em}
\item For each $k\in \mathbb{Z}$, the set $W$ is disjoint union $W=\cup_{Q\in \mathcal D_k} Q$. Moreover, if $k<l, Q\in \mathcal D_l,$ and $R\in \mathcal D_k$, then either $Q\cap R=\emptyset$ or $Q\subset R$.
\item For each $k\in \mathbb{Z}$ and each cube $Q\in \mathcal D_k$, there exists a ball $B(Q)=B(z_Q,r(Q))$ with $z_Q\in W$ such that $A_0^{-k}\leq r(Q)\leq C_0 A_0^{-k}$ and
$W \cap B(Q) \subset Q\subset W \cap 28 B(Q)=W \cap B(z_Q,28 r(Q))$. Further, the collection $\{5B(Q)\}_{Q\in \mathcal D_k}$ is pairwise disjoint.
\item Let $\mathcal D_k^{db}$ denote the collection of cubes $Q$ in $\mathcal D_k$ satisfying the doubling condition
$$\mu (100 B(Q))\leq C_0 \mu(B(Q)).$$
\item
For the non-doubling cubes $Q\in \mathcal D_k\setminus \mathcal D_k^{db}$, we have $r(Q)=A_0^{-k}$ and
\begin{align}
\label{NDprop}
\mu (c B(Q))\leq C_{0}^{-1} \mu(100 cB(Q))
\end{align}
for~all $1\leq c\leq C_0.$
$\mathscr{D}$ will denote the collection of David-Mattila cells.
\end{enumerate}
\end{lem}
In \cite{L15}, the sparse domination was achieved using a suitable stopping time argument with the help of the  following operator
$$M_{T}(f)(x)=\sup_{Q\ni x}||T(f \chi_{\mathbb{R}^n\setminus 3Q})||_{L^\infty(Q)}.$$
The sparse family was then obtained with the help of the recursive use of Calder\'on-Zygmund decomposition. In euclidean setting the Calder\'on-Zygmund decomposition provides disjoint cubes and hence it is suitable for recursive use but Tolsa's decomposition produces ``almost-disjoint" cubes which may not be suitable for recursive use of the decomposition as the overlap grows very rapidly. Thus our cube selection algorithm will be different and will rely on the structure of David-Mattila cells. In the context of upper doubling, geometric doubling metric measure spaces the corresponding grand maximal truncated operator was employed in \cite{vk}. Motivated by that, let us define the following maximal operator for our purpose.
\begin{defn}
For a fixed cell $Q_0$ and $x\in Q_0$, the local bilinear grand maximal truncated operator is defined by
$$\mathcal M_{T, Q_0} (f_1,f_2)(x)=\sup\limits_{ \substack{P\ni x,\\ P\in \mathcal{D}(Q_0)}}\|T(f_1,f_2)-T(f_1 \chi_{30B(P)}, f_{_2} \chi_{30B(P)})\|_{L^\infty(P)}.$$
where $\mathcal D(Q_0)$ denote the collection of David-Mattila cells contained in $Q_0$.
\end{defn}
In order to obtain end-point boundedness for the grand  bilinear maximal truncated operator one would like to observe the difference between $\mathcal M_{T, Q}$ and $T^*$. The following theorem says that $\mathcal M_{T, Q}$ and $T^*$ are equivalent upto the maximal function $\mathcal{M}_{\lambda}$.
\begin{thm}
\label{propo}
Let $Q_0$ be a David-Mattila cell. For functions $f_1,f_2$ supported in $30B(Q_0)$ we have the following inequality
\begin{align}
|\mathcal M_{T, {Q_0}}(f_1,f_2)(x)-T^{*}(f_1,f_2)(x)|\leq (C_{K,\lambda}+||\omega||_{Dini(1)})\mathcal{M}_{\lambda}(f_1,f_2)(x),
\end{align}
for all $x\in Q_0$ and $\mathcal{M}_{\lambda}$ is the maximal function as in \eqref{revmax}. In particular, $\mathcal{M}_{T, Q_0}$ maps $L^1(X, d, \mu)\times L^1(X, d, \mu)$ to $L^{\frac{1}{2},\infty}(X,d,\mu)$ with norm independent of $Q_0$.
\end{thm}
\begin{proof}
Fix $x\in Q_0$ and $r>0$. Observe that for large $r>0$ (say $r>40r(Q_0)$), $\tilde{T}_{r}(f_1, f_2)(x)=0$, as $f_1, f_2$ are supported on $30B(Q_0)$, so we can chose a David-Mattila cell $P\in \mathcal{D}(Q_{0})$ such that $x\in P$ and $r\sim diam(P)$. Denote  $J_k= (2^{k+1} 30B(Q))^{2}\setminus(2^k 30B(Q))^{2}$. Then for any $x'\in P$,
\begin{align*}
&|\tilde{T}_r(f_1,f_2)(x)|-\mathcal M_{T, {Q_0}}(f_1,f_2)(x)\\
& \leq |\tilde{T}_r(f_1,f_2)(x)|-|T(f_1,f_2)(x')-T(f_1 \chi_{30B(P)}, f_2 \chi_{30B(P)})(x')|\\
&\leq |\tilde{T}_r(f_1,f_2)(x)-(T(f_1,f_2)(x)-T(f_1 \chi_{30B(P)}, f_2 \chi_{30B(P)})(x))|\\
&+|(T(f_1,f_2)(x)-T(f_1 \chi_{30B(P)}, f_2 \chi_{30B(P)})(x))-(T(f_1,f_2)(x')-T(f_1 \chi_{30B(P)}, f_2 \chi_{30B(P)})(x'))|\\
&\leq \int_{(30 B(P))^2\triangle S_r(x)}|K(x,y_1,y_2)||f_1(y_1)||f_2(y_2)|d\mu(y_1)d\mu(y_2)\\
&+\int\limits_{X^{2}\setminus (30 B(Q))^{2}}|K(x,y_1,y_2)-K(x',y_1,y_2)||f_1(y_1)||f_2(y_2)|d\mu(y_1)d\mu(y_2)\\
&\leq \int_{(30B(P))^2 \triangle S_r(x)}\min_{i=1,2}\frac{1}{\lambda(x,d(x,y_i)^2}|f_1(y_1)||f_2(y_2)|d\mu(y_1)d\mu(y_2)\\
&+\sum_{k\geq 0}\int\limits_{J_k}|K(x,y_1,y_2)-K(x',y_1,y_2)|f_1(y_1)||f_2(y_2)|d\mu(y_1)d\mu(y_2)\\
&\leq C_{K}\int_{(30B(P))^2 \triangle S_r(x)} \frac{1}{\lambda(x,30 r(P))^2}|f_1(y_1)||f_2(y_2)|d\mu(y_1)d\mu(y_2)\\
&+C_{\lambda}\sum_{k\geq 0}\int\limits_{(2^{k+1} 30B(Q))^{2}}\frac{1}{\lambda(x,2^{k}r(Q))^2}\;\omega\left(\frac{56r(Q)}{2^{k}r(Q)}\right)|f_1(y_1)||f_2(y_2)|d\vec{\mu}(y)\\
&\lesssim (C_{K,\lambda}+||\omega||_{Dini(1)})\mathcal{M}_{\lambda}(f_1,f_2)(x).
\end{align*}
Taking supremum over all $r>0$ and the fact that $|T_{r}(f_1,f_2)(x)-\tilde{T}_{r}(f_1,f_2)(x)|\lesssim \mathcal{M}_{\lambda}(f_1,f_2)$, we get,
$$T^*(f_1,f_2)(x)-\mathcal{M}_{T, Q_0}(f_1,f_2)(x)\lesssim (C_{K,\lambda}+||\omega||_{Dini(1)})\mathcal{M}_{\lambda}(f_1,f_2)(x). $$
Analogously we can prove the other side to conclude
\begin{align}
\label{eqcom}
|T^*(f_1,f_2)(x)-\mathcal{M}_{T, Q_0}(f_1,f_2)(x)|\lesssim (C_{K,\lambda}+||\omega||_{Dini(1)})\mathcal{M}_{\lambda}(f_1,f_2)(x),
\end{align}
for $x\in Q_0$. As both $T^*$ (Theorem~\ref{T*}) and $\mathcal{M}_{\lambda}$ maps $L^1(X, d, \mu)\times L^1(X, d, \mu)$ to $L^{\frac{1}{2},\infty}(X,d,\mu)$ the conclusion follows automatically.
\end{proof}

 \subsection{Proof of Theorem~\ref{sparsethm}} \label{MR}
 
\noindent In this section, we provide proof of Theorem~\ref{sparsethm}. The proof of the sparse domination~(\ref{sparsedom}) is constructive and follows a recursive argument. We shall prove a recursive formula involving the operator $\mathcal M_{T,Q_0}$. Before proceeding further, we fix some notations. Let $\alpha\geq 200$ be a large number and we construct the David-Mattila cells (taking $A_0$ sufficiently large) such that $30B(Q)\subset 30B(\hat{Q})$ where $\hat{Q}$ denotes the parent of $Q$. For a cell $Q$, we denote
\begin{align*}
A(\vec{f},Q)&:=\prod_{i=1}^{2}\frac{1}{\mu(\alpha B(Q))}\int_{30 B(Q)}|f_i|d\mu\\
\Upsilon(Q)&:=\frac{{\mu(\alpha B(Q))}^2}{{\lambda(z_{Q},\alpha r(Q))}^2}~~~\mbox{and}~~\tilde{\mathcal{M}}(f_1,f_2)(x)=\sup\limits_{\substack{Q\in \mathscr{D}\\~Q\ni x}}A(\vec{f},Q).
\end{align*}
It is easy to see that $\Upsilon(Q)\leq 1$ and $\tilde{\mathcal{M}}$ is bounded from $L^{1}(X,d,\mu)\times L^1(X,d,\mu)$ to $L^{\frac{1}{2},\infty}(X,d,\mu)$ (by Vitali covering lemma). First we provide the following relation between
the grand maximal truncated operator in two consecutive scales of David-Mattila cells and this will be useful to provide the recursive relation.
\begin{prop}[Consecutive scales]
\label{c6conscales}
Let $Q$ be any David-Mattila cell. Then
 \begin{align}
\label{c6eqre}
\mathcal M_{T,\hat{Q}}(f_1 \chi_{30 B(\hat{Q})},f_2 \chi_{30 B(\hat{Q})})(x)\lesssim C\Upsilon(\hat{Q})A(\vec{f},\hat{Q})+\mathcal M_{T,Q}(f_{1}\chi_{30B(Q)},f_2 \chi_{30 B(Q)})(x)
\end{align}
for all $x\in Q$.
\end{prop}
\begin{proof}
The proof simply follows from the fact that for every cell $P\in \mathcal D(\hat Q)$ with $x\in P$, we have
\begin{align*}
&|T(f_1 \chi_{30 B(\hat{Q})},f_2\chi_{30 B(\hat{Q})})(y)-T(f_1 \chi_{30 B(\hat{Q})} \chi_{30B(P)}, f_2 \chi_{30 B(\hat{Q})} \chi_{30B(P)})(y)|\\
&=\vert\int_{30B(\hat{Q})^2\setminus30B(P)^2}K(y,z_1,z_2)f_1(z_1)f_2(z_2)d\mu(z_1)d\mu(z_2)|\\
&=\vert\int_{30B(Q)^2\setminus30B(P)^2}K(y,z_1,z_2)f_1(z_1)f_2(z_2)d\mu(z_1)d\mu(z_2)\\
&+\int_{30B(\hat{Q})^2\setminus30B(Q)^2}K(y,z_1,z_2)f_1(z_1)f_2(z_2)d\mu(z_1)d\mu(z_2)\vert\\
&\lesssim \mathcal M_{T,{Q}}(f_{1}\chi_{30B(Q)},f_2 \chi_{30B(Q)})(x)+ \frac{1}{\lambda(x,r(Q))^2}\int_{30B(\hat{Q})^2}|f_1(z_1)||f_2(z_2)|d\mu(z_1)d\mu(z_2)\\
&\leq C\Upsilon(\hat{Q})A(\vec{f},\hat{Q})+\mathcal M_{T,Q}(f_{1}\chi_{30B(Q)},f_2 \chi_{30B(Q)})(x).
\end{align*}
Now taking supremum over all such cells concludes the proof. 
\end{proof}
The main argument in the proof of Theorem~\ref{sparsethm}, depends on the following cube selection algorithm. In \cite{L15}, all the cubes are doubling but in our case, we have to face ``good" as well as ``bad" cubes. In the following selection procedure if we face a ``bad" cube we keep on subdividing it until we hit ``good" cubes and this fact is guaranteed due to the construction of David-Mattila cells. Also in this procedure, we assure the fact that the constants $\Upsilon(Q)$ decay exponentially which helps us to sum. This fact follows from the Lemma~3.6
in \cite{vk}. We state it for convenience. 
\begin{lem}
Let $Q_0=\hat{Q_1}\supset Q_1=\hat{Q_2}\supset Q_{2}\dots$ and $Q_1,Q_2,\dots$ are non-doubling, then we have $\Upsilon(Q_k)\leq C_{0}^{-k \frac{l_{0}}{2}}\Upsilon(Q_0)$, where $l_0$ is the maximal positive integer such that $100^{l_{0}}\alpha\leq C_0$.
\end{lem}
\begin{proof}
We will only prove $\Upsilon(Q_1)\leq C_{0}^{-\frac{l_{0}}{2}}\Upsilon(Q_0)$, then the proof follows from recursion. Let $l_0$ be the maximal integer such that $100^{l_{0}}\alpha\leq C_0$. As $Q_1$ is a bad cube, using \eqref{NDprop}, we get
$$\mu(\alpha B(Q_1))\leq C_{0}^{-l_0}\mu(100^{l_0}\alpha B(Q_1))\leq C_{0}^{-l_0}\mu(\alpha B(\hat{Q}_1))=C_{0}^{-l_0}\mu(\alpha B(Q_0)),$$
where the last inequality requires $A_0$ to be very large, for example $A_0=CC_{0}^{100}$ will do.
Now using the doubling property of $\lambda$ we conclude, $\lambda(z_{Q_0}, \alpha r(Q_0))\leq C_{\lambda}^{\left \lceil{\log_{2}A_{0}}\right\rceil}\lambda(z_{Q_1}, \alpha r(Q_1))$. Now if we chose $C_0$ such that $C_{0}^{\frac{l_0}{2}}>C_{\lambda}^{\left \lceil{\log_{2}A_{0}}\right\rceil}$, we get $\Upsilon(Q_1)\leq C_{0}^{-\frac{l_0}{2}}\Upsilon(Q_0)$ and we are done.
\end{proof}
\begin{lem}
\label{c6lemma1}
Let $Q_0\in \mathscr{D}^{db}$ and $f_1,f_2$ be integrable functions with $\supp(f_j)\subseteq 30B(Q_0)$ for $j=1,2$. Then we can find a subset $E \subset Q_0$, families of pairwise disjoint David-Mattila cells $\mathcal{L}_n(Q_0)\subset \mathscr{D},\;\;n=1,2,3,\dots$ contained in $E$, and a family of pairwise disjoint cells $\mathcal{F}(Q_0)\subset \mathscr{D}^{db}$ contained in $E$ such that $\mu(E)\leq \frac{1}{2}\mu(Q_0)$ and for every $P\in\mathcal{F}(Q_0),~ Q\in\mathcal{L}_n(Q_0)$ either $P\subset Q$ or $P\cap Q=\emptyset,$ and
\begin{align}
\nonumber\mathcal{M}_{T,Q_0}(f_1 \chi_{30 B(Q_0)},f_2 \chi_{30B(Q_0)})\chi_{Q_0} &\leq \sum\limits_{P\in \mathcal{F}(Q_0)}\mathcal{M}_{T,P}(f_1 \chi_{30 B(P)},f_2 \chi_{30 B(P)})\chi_{P}\\
&+C A(\vec{f},Q_0)+ C\sum\limits_{n=1}^\infty 100^{-n}\sum\limits_{Q\in \mathcal{L}_{n}(Q_0)}A(\vec{f},Q) \chi_{Q}.
\end{align}
\end{lem}

The following diagram pictorially depicts our selection algorithm.\\
\\
\tikzstyle{decision} = [diamond, draw, fill=white!20, 
text width=3cm, text badly centered, node distance=3cm, inner sep=0pt]
 \tikzstyle{block} = [rectangle, draw, fill=white!20, 
text width=3cm, text centered, rounded corners, minimum height=4em]
\tikzstyle{line} = [draw, -latex']
\tikzstyle{cloud} = [draw, ellipse,fill=red!20, node distance=3cm,
minimum height=5em]
\begin{small}
 \begin{tikzpicture}[node distance = 2.5cm, auto]
\node [block] (sume) {Recursive Diagram};
\node [block, below of=sume] (decide) {1.Start with a doubling cell $Q_0$};
\node [block, below right =1cm and 1 cm of decide, node distance=5cm] (no) {3.Bad cells~$\mathcal{L}(Q_0)$};
\node [block, below left = 1cm and 1cm of decide, node distance=5cm] (yes) {2.Good cells~$\mathcal{F}(Q_0)$};
\node [block, below of= no] (first) {Bad child $\mathcal{L}_n(Q_0)$};
\node [block, below of= first] (firstgood) {Good child};

\path [line] (sume) -- node[pos=0.5](z){}(decide);
\path [line] (decide) -| node[pos=0.2,above] {Good} (yes);
\path [line] (yes.west) --++(-1cm,0) |- (z);    
\path [line] (decide) -| node[pos=0.2,above] {Bad} (no);
\path [line] (no) -- node[pos=0.3,below left=.01cm]{Recursive use of Lemma~\ref{c6conscales}}(first);
\path [line] (firstgood.west) --++(-5cm,0) |- (yes);
\path [line] (first) -- node[pos=0.3,below left=.01cm]{Final iteration of Lemma~\ref{c6conscales}}(firstgood);
\end{tikzpicture}
\end{small}
\begin{proof}
Let $Q_{0}\in \mathscr{D}^{db}$, now the boundedness of the operators $\mathcal M_{T,Q_0}~\mbox{and}~ \tilde{\mathcal{M}}$ from $L^1(X, d, \mu)\times L^{1}(X, d, \mu)$ to $L^{\frac{1}{2}, \infty}(X, d, \mu)$ imply that for large enough $\Theta>0$, we have  $\mu(E)\leq \frac{1}{2}\mu(Q_0)$,
where
$$E:=\{ x\in Q_0:\max(\mathcal{M}_{T,Q_0}(f_1,f_2)(x), \tilde{\mathcal{M}}(f_1,f_2)(x))> \Theta A(\vec{f},Q_0)\}.$$
In \cite{L15}, performing Calder\'on-Zygmund decomposition to a set similar to $E$ produces the required sparse family but as we have to face ``good" as well as ``bad" cubes in our case our selection procedure will be different. Let $\mathcal{L}_0(Q_0)$ is the maximal David-Mattila cells inside $E$, this implies $\sum_{Q\in \mathcal{L}_0(Q_0)}\mu(Q)\leq \frac{1}{2} \mu(Q_0)$. These cells may not be doubling. For $x\in E\setminus \bigcup\limits_{\mathcal{L}_0}Q$ we have $\mathcal{M}_{T,Q_0}(f_1,f_2)(x)\leq \Theta A(\vec{f},Q_0)$ and we are in a nice situation.  For a cell $Q\in \mathcal{L}_0(Q_0)$ and $x\in Q$, the maximality of $Q$ implies
$$ess\sup\limits_{y\in \tilde{Q}}|T(f_1,f_2)(y)-T(f_1 \chi_{30B(\tilde{Q})}, f_2 \chi_{30B(\tilde{Q})})(y)|\leq \Theta A(\vec{f},Q_0),$$
for all $Q\subsetneq \tilde{Q}\subseteq Q_0$. Thus we get
\begin{align}
\mathcal M_{T,Q_0}(f_1,f_2)(x)\leq \Theta A(\vec{f},Q_0)+\mathcal M_{T,\hat{Q}}(f_1 \chi_{30B(\hat{Q})},f_2 \chi_{30 B(\hat{Q})})(x).
\label{eqrecursive}
\end{align}
Applying Proposition~\ref{c6conscales} to \eqref{eqrecursive} and the fact that $\Upsilon(\hat{Q})\leq 1$, yield 
\begin{align}
\mathcal M_{T,Q_0}(f_1,f_2)(x)\leq K A(\vec{f},Q_0)+\mathcal M_{T,Q}(f_{1} \chi_{30 B(Q)},f_2 \chi_{30 B(Q)})(x)
\label{eqre1}
\end{align}
for some bigger constant $K$. Now we keep on iterating \eqref{c6eqre} to $Q$ (atmost a finite number of times) until we reach a ``good" cube containing $x$ and this is guaranteed by the fact that $\mu(Q\setminus \bigcup\limits_{R\in S(Q)} R)=0$, where $S(Q)$ denotes  the set of maximal good David-Mattila cells in $Q$ (see Lemma 5.28 and 5.31 in \cite{dm}). Now the construction of David-Mattila cells ensure that this process will terminate after a finite number of steps i.e., there can be only finite chain of nondoubling cubes $Q_1\supset Q_2\dots \supset Q_{J}$ such that $Q_j\in \mathcal{L}_j$ and $x\in Q_j$. We have already pointed out that for each such $Q_j$, $\Upsilon(Q_j)\lesssim 100^{-j}$ if we take $C_0$ to be large enough. We iterate this procedure to each bad cell in the collection $\mathcal{L}_0(Q_0)$. Next, we separate the cells with doubling property and set 
$$\mathcal{F}(Q_0)=\lbrace Q\in \mathcal{L}_0{}(Q_0): Q \;\text{is~doubling}\rbrace~~\text{and}~\mathcal{L}_{1}(Q_0)=\mathcal{L}_{0}(Q_0)\setminus \mathcal{F}.$$
Now for a cell $Q\in \mathcal{D}$ such that $\hat{Q}\in \mathcal{L}_1(Q_0)$, if $Q$ is doubling we select it to the collection $\mathcal{F}(Q_0)$, otherwise it goes to $\mathcal{L}_2(Q_0)$. The exponential decay of the quantities $\Upsilon(Q)$ allow us to sum the intermediate levels and a recursive application of this procedure concludes the proof of the lemma.
\end{proof}

Recall Theorem~\ref{propo}, and note that in order to prove sparse domination of the operator $T^*$, we need sparse domination for the maximal operator $\mathcal{M}_{\lambda}$. The next lemma establishes this sparse domination. $\mathcal{M}_{\lambda}$ is dominated by the dyadic version of it, namely 
$\mathcal M_{\lambda}^{\mathscr{D}}$ which is defined as 
$$\mathcal M_{\lambda}^{\mathscr{D}}(\vec{f})(x):=\sup\limits_{x\in P, P\in \mathscr{D}}\Upsilon(P)A(\vec{f}, P).$$
We require the localized version 
$\mathcal M_{\lambda}^{\mathscr{D},Q_0}$ of $\mathcal M_{\lambda}$ for doubling cell $Q_0$, which is defined by taking the supremum over cells contained in $\mathscr D(Q_0)$, that is,
\begin{align*}
\mathcal{M}_{\lambda}^{\mathscr{D}, Q_0}(\vec{f})(x)&=\sup\limits_{x\in P,~ P\in \mathscr{D}(Q_0)}\prod_{i=1}^{2} \frac{1}{\lambda(z_{P},r(P))} \int_{30 B(P)} |f_{i}|  d\mu\\
&=\sup\limits_{x\in P,~ P\in \mathscr{D}(Q_0)}\Upsilon(P)A(\vec{f}, P).
\end{align*}
\begin{lem}
Let $Q_0\in \mathscr{D}^{db}$ and $f_1,f_2$ be integrable functions with $\supp(f_j)\subseteq 30B(Q_0)$ for $j=1,2$. Then we can find a subset $E\subset Q_0$, families of pairwise disjoint cells $\mathcal{L}_n(Q_0)\subset \mathscr{D},\;\;n=1,2,\dots$ contained in $E$, and a family of pairwise disjoint ``good" cells $\mathcal{F}(Q_0)\subset \mathscr{D}^{db}$ contained in $E$ such that $\mu(E)\leq \frac{1}{2}\mu(Q_0)$ and for every $P\in\mathcal{F}(Q_0),~ Q\in\mathcal{L}_n(Q_0)$ either $P\subset Q$ or $P\cap Q=\emptyset,$ and
\begin{align}
\nonumber~\mathcal{M}_{\lambda}^{\mathscr D, Q_0}(\vec{f})\chi_{Q_0} &\leq \sum\limits_{P\in \mathcal{F}(Q_0)}\mathcal{M}_{\lambda}^{\mathscr{D}, P}(\vec{f})\chi_{P}+C A(\vec{f},Q_0)\\
&+C\sum\limits_{n=1}^\infty 100^{-n}\sum\limits_{Q\in \mathcal{L}_{n}(Q_0)}A(\vec{f},Q)\chi_{Q}.
\end{align}
\end{lem}
The proof of this lemma is similar to the previous lemma. Thus we skip the proof.
\qed

\noindent\textbf{Proof of Theorem~\ref{sparsethm}:}
To obtain the sparse domination Theorem~\ref{sparsethm} we need sparse domination for $\mathcal{M}_{T,Q_0}$ and $\mathcal{M}_{\lambda}$ (see Theorem~\ref{propo}). We only mention it for $\mathcal{M}_{T, Q_0}$.

 As the upper doubling condition ensures there are plenty of ``large-doubling" balls, we start with a ball $D$ such that $\mu(100D)\leq C_0 \mu(D)$, also without loss of generality we may assume that the ball has radius $C_0$. Then the construction of David-Mattila cells ensures that there is a doubling-cell $Q_0$ such that $X'\subset Q_0$. Let $\mathcal{G}_0^0=\lbrace Q_0\rbrace$. Applying Lemma~\ref{c6lemma1} we get collection of good cubes $\mathcal{G}^{Q_0}$ as well as bad cubes $\mathcal{L}^{Q_0}_n$, for $n=1,2,\dots$. Define $\mathcal{G}^1_{n}=\mathcal{L}^{Q_0}_n$, for $n=1,2,\dots$ and $\mathcal{G}_{0}^1=\mathcal{G}^{Q_0}$. Now we initiate the recursive application of Lemma~\ref{c6lemma1}. If $\mathcal{G}_{0}^{k}$ is constructed, now for each doubling-cube $P\in \mathcal{G}_0^{k}$ the above-mentioned process (Lemma~\ref{c6lemma1}) produces non-doubling cubes $\mathcal{L}_n^{P}$ for $n=1,\dots$ and doubling cubes $\mathcal{G}^{P}$. 
 
We separate the bad cubes according to the stage they are appearing, i.e., define for $n=1,2,\dots$,~ $\mathcal{G}_n^{k+1}=\bigcup\limits_{P\in {\mathcal{G}_0^{k}}}\mathcal{L}_n^{P}$ and $\mathcal{G}_{0}^{k+1}=\bigcup\limits_{P\in \mathcal{G}_0^{k}}\mathcal{G}^{P}$. Now one can observe that at each step we are ensuring the following, for $Q\in \mathcal{G}_n^{k}$ we have$$\sum\limits_{R\in \mathcal{G}_n^{k+1};R\subset Q}\mu(R)\leq \frac{1}{2}\mu(Q).$$
$\mathcal{G}_n=\bigcup\limits_{k}\mathcal{G}^{k}_n$ is our desired sparse families. This completes the proof of Theorem~\ref{sparsethm}.\\
\qed 
\subsection{Weighted estimates}
\label{weighted estimates}
\begin{prop}
\label{sparseprop}
Let $1<p_1,p_2<\infty$ be such that $\frac{1}{p}=\sum\limits_{i=1}^{2}\frac{1}{p_i}$. Then for any multiple weight $\vec{w}$ we have $$\|\mathcal{A_S}(\vec{f})\|_{L^p(v_{\vec{w}})} \lesssim C_{w} \prod_{i=1}^{2}\|f_i\|_{L^{p_i}(w_i)},$$ with the implicit constant independent of $\vec{w}$ and $C_w$ as in Corollary~\ref{c6mainresult}.
\end{prop}
\begin{proof}
We first consider the case $\frac{1}{2}< p\leq 1$. Let us assume $ p_1=\min\lbrace p_1,p_2\rbrace $ and $\vec{f}_{\sigma}$ denotes the tuple $(f_1 \sigma_1, f_2 \sigma_2)$. Now,
\begin{small} 
 \begin{align*}
\int_{X}\mathcal{A_S}(\vec{f}_{\sigma})^p v_{\vec{w}}d\mu 
&\leq \sum_{Q\in \mathcal{S}}\left(\prod_{i=1}^{2}\frac{1}{\mu(\alpha B(Q))}\int_{30B(Q)}|f_i|\sigma_{i}d\mu\right)^p v_{\vec{w}}(Q)\\
&\leq (\sup_{Q}K_{Q})\sum_{Q\in \mathcal{S}}\frac{\mu(Q)^{2p(p'_1-1)}}{v_{\vec{w}}(Q)^{p'_{1}-1}\prod\limits_{i=1}^{2}\sigma_{i}(\alpha B(Q))^{\frac{p p'_1}{p'_{i}}}}\left(\prod_{i=1}^{2}\int_{30B(Q)}|f_i|\sigma_{i}d\mu\right)^p,
 \end{align*}
\end{small}
where $K_{Q}=\frac{v_{\vec{w}}(Q)^{p'_{1}}\prod\limits_{i=1}^{2}\sigma_{i}(\alpha B(Q))^{\frac{p p'_{1}}{p'_{i}}}}{\mu(\alpha B(Q))^{2p}\mu(Q)^{2p(p'_{1}-1)}}$. Now using the sparseness, we observe that
\begin{small}
\begin{align*}
\mu(Q)^{2p(p'_{1}-1)}\lesssim \mu(E(Q))^{2p(p'_{1}-1)}\leq v_{\vec{w}}(Q)^{p'_{1}-1}\prod\limits_{i=1}^{2}\sigma_{i}(E_Q)^{\frac{p(p'_{1}-1)}{p'_{i}}}.\\
v_{\vec{w}}(Q)^{p'_{1}-1}\leq v_{\vec{w}}(Q)^{p'_{1}-1}~
{\text~and}~~\sigma_{i}(E_Q)^{\frac{p(p'_{1}-1)}{p'_{i}}-\frac{p}{p_i}}\leq \sigma_{i}(Q)^{\frac{pp'_{1}}{p'_{i}}-p}.
\end{align*}
\end{small}
As $\alpha\geq 200$, the Vitali covering lemma ensures the boundedness of the modified uncentered maximal function $\tilde{M}_{\sigma}f(x):=\sup_{B\ni x}\frac{1}{\sigma(\alpha B)}\int_{30B}|f|\sigma d\mu$ from $L^p(\sigma)$ to itself for $1<p\leq \infty$, where $\sigma(\alpha B)$ denotes $\int_{\alpha B} \sigma d\mu$. Using this fact together with the above estimates we get,
\begin{small}
\begin{align*}
\int_{X}\mathcal{A_S}(\vec{f}_{\sigma})^p v_{\vec{w}}d\mu 
&\leq  (\sup_{Q}K_{Q})\sum_{Q\in \mathcal{S}}\prod\limits_{i=1}^{2}\left(\frac{1}{\sigma_{i}(\alpha B(Q))}\int_{30B(Q)}|f_i|\sigma_{i}d\mu\right)^{p}\sigma_{i}(E_Q)^{\frac{p}{p_i}}\\
&\leq (\sup_{Q}K_{Q}) \prod\limits_{i=1}^{2}\Big\lbrace{\sum_{Q\in \mathcal{S}}\left(\frac{1}{\sigma_{i}(\alpha B(Q))}\int_{30B(Q)}|f_i|\sigma_{i} d\mu\right)^{p_i}\sigma_{i}(E_Q)\Big\rbrace}^{\frac{p}{p_i}}\\
&\leq (\sup_{Q}K_{Q})\prod\limits_{i=1}^{2}\left(\sum_{Q\in \mathcal{S}}\int_{E_Q}\tilde{M}_{\sigma_{i}d\mu}(f_i)^{p_i}\sigma_{i}d\mu\right)^{\frac{p}{p_i}}\\
&\leq (\sup_{Q}K_{Q})\prod\limits_{i=1}^{2}||\tilde{M}_{\sigma_{i}d\mu}(f_i)||_{L^{p_i}(\sigma_{i}d\mu)}^p\\
&\leq (\sup_{Q}K_{Q}) \prod\limits_{i=1}^{2}\|f_i\|_{L^{p_i}(\sigma_{i}d\mu)}^p.
\end{align*}
\end{small}
Hence, $\|\mathcal{A_S}(\vec{f})\|_{L^p(v_{\vec{w}})}\leq C_{w}\prod_{i=1}^{2}\|f_i\|_{L^{p_i}(w_i)}$,
where $C_{w}=\sup\limits_{Q}\frac{v_{\vec{w}}(Q)^{\frac{p'_{1}}{p}}\prod\limits_{i=1}^{2}\sigma_{i}(\alpha B(Q))^{\frac{p'_{1}}{p'_{i}}}}{\mu(\alpha B(Q))^{2}\mu(Q)^{2(p_{1}^{'}-1)}}$.

If the underlying measure is doubling then $\mu(\alpha B(Q))$ and $\mu(B(Q))$ are comparable, we observe,
\begin{small}
\begin{equation*}
 C_{w}\lesssim_{\alpha}  \frac{v_{\vec{w}}(\alpha B(Q))^{\frac{p'_1}{p}}\prod\limits_{i=1}^{2}\sigma_{i}(\alpha B(Q))^{\frac{p p'_{1}}{p'_{i}}}}{\mu(\alpha B(Q))^{2}\mu(\alpha B(Q))^{2(p'_{1}-1)}}\simeq [\vec{w}]_{A_{\vec{P}}(\mu)}^{\frac{p'_{1}}{p}}
\end{equation*}
\end{small}
Therefore, for doubling measure, we have 
$\|\mathcal{A_S}(\vec{f})\|_{L^p(v_{\vec{w}})} \lesssim [\vec{w}]_{A_{\vec{P}}(\mu)}^{\frac{p_{1}^{'}}{p}}\prod_{i=1}^{2}\|f_i\|_{L^{p_i}(w_i)}$.

For  $p>1$, we use the duality . Let  $g\in L^{p'}(v_{\vec{w}})$  be a non-negative function and $\vec{f}_{\sigma}=(f_1 \sigma_1, f_2 \sigma_2)$, 
\begin{small}
\begin{align*}
&\int_{X}\mathcal{A_S}(\vec{f}_{\sigma})g v_{\vec{w}} d\mu\\
&\leq \sum_{Q\in \mathcal{S}}\int_{Q} g v_{\vec{w}} d\mu \cdot \prod_{i=1}^2\frac{1}{\mu(\alpha B(Q))}\int_{30B(Q)}|f_i|\sigma_{i} d\mu\\
&=\sum_{Q\in \mathcal{S}}\frac{v_{\vec{w}}(Q)\prod\limits_{i=1}^{2}\sigma_{i}(200B(Q))}{\mu(\alpha B(Q))^2}\frac{1}{v_{\vec{w}}(Q)}\int_{Q}g v_{\vec{w}}d\mu\prod\limits_{i=1}^{2}\frac{1}{\sigma_{i}(200B(Q))}\int_{30B(Q)}|f_i|\sigma_{i}d\mu\\
& \lesssim \sup\limits_{Q}\frac{v_{\vec{w}}(Q)\prod\limits_{i=1}^{2}\sigma_{i}(200B(Q))}{\mu(\alpha B(Q))^2 {v_{\vec{w}}(E_Q)}^{\frac{1}{p^{'}}}\prod\limits_{i=1}^2{\sigma_i(E_Q)}^{\frac{1}{p_i}}}\left[\sum_{Q\in \mathcal{S}}(\frac{1}{v_{\vec{w}}(Q)}\int_{Q}g v_{\vec{w}}d\mu)^{p^{'}}v_{\vec{w}}(E_Q)\right]^{\frac{1}{p^{'}}}\\
&\prod\limits_{i=1}^2\left[\sum_{Q\in \mathcal{S}}(\frac{1}{\sigma_{i}(200B(Q))}\int_{30B(Q)}|f_i|\sigma_{i}d\mu)^{p_i}\sigma_i(E_Q)\right]^{\frac{1}{p_i }}.
\end{align*}
\end{small}
Denote $ L_Q=\frac{v_{\vec{w}}(Q)\prod\limits_{i=1}^{2}\sigma_{i}(200B(Q))}{\mu(\alpha B(Q))^2 {v_{\vec{w}}(E_Q)}^{\frac{1}{p'}}\prod\limits_{i=1}^2{\sigma_i(E_Q)}^{\frac{1}{p_i}}}$.
Hence we obtain
\begin{small}
\begin{align*}
\int_{X}\mathcal{A_S}(\vec{f}_{\sigma})g v_{\vec{w}} d\mu
&\lesssim (\sup\limits_{Q}L_Q) \|M_{v_{\vec{w}}d\mu}^{\mathscr{D}}g\|_{L^{p'}(v_{\vec{w}}d\mu)}\prod_{i=1}^{2}\|\tilde{M}_{\sigma_{i}d\mu}(f_i)\|_{L^{p_i}(\sigma_i)}\\
&\lesssim (\sup\limits_{Q}L_Q)\|g\|_{L^{p'}(v_{\vec{w}}d\mu)}\prod_{i=1}^2\|f_i\|_{L^{p_i}(\sigma_i)}.
\end{align*}
\end{small}
Thus using duality we get $\|\mathcal{A_S}(\vec{f})\|_{L^p(v_{\vec{w}})}\leq (\sup\limits_{Q}L_{Q})\prod_{i=1}^{2}\|f_i\|_{L^{p_i}(w_i)}$.
Now there are two cases.\\

\textbf{Case1}: $p\geq\max\limits_{i}p'_i$.\\
Now observe that, $$\mu(Q)^{2(p-1)}\lesssim \mu(E_Q)^{2(p-1)}\lesssim v_{\vec{w}}(E_Q)^{\frac{p-1}{p}}\prod_{i=1}^{2}\sigma_{i}(E_Q)^{\frac{p-1}{p'_i}},$$
 and as $E_Q\subset Q$, we have $\sigma_{i}(E_Q)^{\frac{p}{p'_i}}\leq\sigma_i(200B(Q))^{\frac{p}{p'_i}-1}\sigma_i(E_Q)$ for all $i=1,2$.

Using these facts we get $\sup\limits_{Q}L_{Q}\leq C_{w}$
where, $$C_{w}=\sup_{Q}\frac{v_{\vec{w}}(Q)\prod\limits_{i=1}^{2}\sigma_{i}(200 B(Q))^{\frac{p}{p'_i}}}{\mu(\alpha B(Q))^2 \mu(Q)^{2(p-1)}}.$$
Similar to the previous case if we restrict ourselves to doubling measures one can see that $C_{w}\simeq [\vec{w}]_{A_{\vec{P}}}$. Therefore, for the doubling measure we get
 $$\|\mathcal{A_S}(\vec{f}_{\sigma})\|_{L^p(v_{\vec{w}})}\lesssim [\vec{w}]_{A_{\vec{P}}(\mu)}\prod_{i=1}^{2}\|f_i\|_{L^{p_i}(\sigma_i)}.$$ 
 
\textbf{Case2}: Finally we consider the general case of exponents. Hence without loss of generality we may assume that $p'_1\geq \max\{p,p'_2\}$. Observe the following facts:
\begin{align*}
\mu(Q)^{2(p'_{1}-1)}\lesssim \mu(E_Q)^{2(p'_{1}-1)}\lesssim v_{\vec{w}}(E_Q)^{\frac{p'_{1}-1}{p}}\sigma_{1}(E_{Q})^{\frac{1}{p_1}}\sigma_{2}(E_Q)^{\frac{p'_{1}-1}{p_2^{'}}} \\ \text{and}\;\; \sigma_{2}(E_Q)^{\frac{p'_{1}}{p'_{2}}-1}\leq \sigma_{2}(200B(Q))^{\frac{p'_{1}}{p'_{2}}-1}.
\end{align*}
Using the above estimates it is easy to see that 
\begin{align*}
C_{w}=\sup_{Q}L_{Q}&\leq \sup_{Q} \frac{v_{\vec{w}}(Q)^{\frac{p'_{1}}{p}}\sigma_{1}(200B(Q))\sigma_{2}(200B(Q))^{\frac{p'_{1}}{p'_{2}}}}{\mu(\alpha B(Q))^2 \mu(Q)^{2(p'_{1}-1)}}\\
&\simeq [\vec{w}]_{A_{\vec{P}}}^{\frac{p'_1}{p}},
\end{align*}
if the underlying measure is doubling. This completes the proof of Proposition~\ref{sparseprop}.
\end{proof}
Combining Theorem~\ref{sparsethm} and Proposition~\ref{sparseprop} yields the  quantitative weighted estimates for bilinear Calder\'on-Zygmund operators on non-homogeneous setting and therefore concludes the proof of Corollary~\ref{c6mainresult}. 
\section{End-point estimates}
\label{endpointresults}
\subsection{Proof of Theorem~\ref{endpoint}}
In this subsection we prove Theorem~\ref{endpoint}. As mentioned earlier this result was proved on  $\mathbb{R}^n$ with polynomial growth measures in \cite{marti}. One of the difficulty in our case is to get the required estimates after decomposing the integrals over annular regions and to overcome that we shall need some extra lemmas. The Calder\'on-Zygmund decomposition in this setting was given by Tolsa for polynomial growth measures on $\mathbb{R}^n$ (see \cite{T3}) then in \cite{Bui} it was extended to non-homogeneous setting. For $\zeta,\tau>1$, we say a ball $B$ is $(\zeta,\tau)$ doubling if $\mu(\zeta B)\leq \tau \mu(B)$. We will strictly follow the Calder\'on-Zygmund decomposition from \cite{Bui} (Lemma~6.1). We also need the following lemma which helps us to use the size condition of the kernel appropriately.
\begin{lem}
\label{bglemma}
Let $(X,d,\mu)$ be a upper doubling, geometric doubling metric measure space. Let $x\in X,~\epsilon>0$. Then
$$\int_{d(x,z)>\epsilon}\frac{1}{\lambda(x,d(x,z))^2}d\mu(z)\lesssim \frac{1}{\lambda(x,\epsilon)}.$$
\end{lem}
\begin{proof} 
Denote $\epsilon_0=\epsilon$. Let $\epsilon_1$ be the smallest $2^k\epsilon$ such that $\lambda(x,2^k\epsilon_0)>2\lambda(x, \epsilon_0)$ where $k\in \mathbb{N}$. $\epsilon_2$ be the smallest $2^k\epsilon_1$ such that $\lambda(x,2^k\epsilon_1)>2\lambda(x, \epsilon_1)$ holds and so on. Let $k_i$ be the smallest corresponding to $\epsilon_i$, i.e, $\epsilon_{i+1}=2^{k_i}\epsilon_i$ and $\lambda(x, 2^{k_i}\epsilon_{i})>2\lambda(x, \epsilon_i)$, also $\lambda(x, 2^{k_{i}-1}\epsilon_{i})\leq 2\lambda(x, \epsilon_i)$. Thus collecting all this
\begin{align*}
\int_{d(x,z)>\epsilon}\frac{1}{\lambda(x,d(x,z))^2}d\mu(z)&\leq\sum_{i}\int_{B(x,\epsilon_{i+1})\setminus B(x, \epsilon_i)}\frac{1}{\lambda(x,d(x,z))^2}d\mu(z)\\&\lesssim \sum_{i}\frac{\mu(B(x,\epsilon_{i+1}))}{\lambda(x, \epsilon_{i+1})^2}\lesssim \sum_{i}\frac{1}{2^i\lambda(x,\epsilon)}\lesssim \frac{1}{\lambda(x,\epsilon)}.
\end{align*}
If for some $\epsilon_{i_0}$ there is no $k$ satisfying our algorithm, then $\lambda(x, 2^k\epsilon_{i_0})\leq 2\lambda(x,\epsilon_{i_0})$ for all $k\geq 1$. Then $\mu(X)=\lim_{k\rightarrow \infty}\mu(B(x, 2^k\epsilon_{i_0}))\leq 2\lambda(x,\epsilon_{i_0})<\infty$. Then
\begin{align*}
&\int_{d(x,z)>\epsilon}\frac{1}{\lambda(x,d(x,z))^2}d\mu(z)\\
&\leq\sum_{i\geq 0}^{i_{0}-2}\int_{B(x,\epsilon_{i+1})\setminus B(x, \epsilon_i)}\frac{1}{\lambda(x,d(x,z))^2}d\mu(z)+\int_{X\setminus B_{i_{0}-1}}\frac{1}{\lambda(x,d(x,z))^2}d\mu(z)\\
&\lesssim \sum_{i}^{i_{0}-2}\frac{1}{2^i\lambda(x,\epsilon)}+\frac{2}{\lambda(x,\epsilon)}\\
&\lesssim C_{\mu}\frac{1}{\lambda(x,\epsilon)}.
\end{align*}
\end{proof}
The following lemma is  from \cite{Bui}. One needs some minor correction in the proof given in \cite{Bui}.
\begin{lem}
Let $(X,d,\mu)$ be a upper doubling, geometric doubling space. Let $Q\subset R$ be concentric balls such that there are no $(\zeta, \tau)$- doubling balls (with $\tau>C_{\lambda}^{\log_2\zeta}$) of the form $\zeta^k Q, k\geq 0,$ with $Q\subset \zeta^k Q\subset R$, then
\begin{align*}
\int_{R\setminus Q}\frac{1}{\lambda(x, d(x, c_{Q}))}d\mu(x)\leq C,
\end{align*}
where $C$ depends only on $\zeta, \tau,$ and $\mu$.
\label{tolsa}
\end{lem}
The decomposition and idea of the proof are based on Proposition~5.2 in \cite{marti} except for certain estimates. For convenience, we are including a detailed sketch of the proof.\\

Before going into the proof let us collect some important ingredients which will be helpful. Let $f_1,f_2\in L^1(\mu)$ with $\|f_i\|_{L^1(X,d,\mu)}=1$ with $i=1,2$ and $t>0$. 

We perform Calder\'on-Zygmund decomposition (Lemma~6.1 in \cite{Bui}) to $f_1$ at height $t^\frac{1}{2}$ and get disjoint balls $\{Q_{1,i}\}$  and $R_{1,i}$ is the smallest $(3\times 6^2, C_{\lambda}^{\log_{2} 3\times 6^2+1})$-doubling ball of the form $\{(3\times 6^2)^k Q_{1,i}\}_{k\geq 1}$ satisfying the following properties. We write $f_1=g_1+\beta_1=g_1+\sum_{i}\beta_{1,i}$, where,

\begin{enumerate}[i)]
\setlength\itemsep{1.2em}
\item 
$\frac{1}{\mu(6^2Q_{1,i})}\int_{Q_{1,i}}|f_1|d\mu\gtrsim t^{\frac{1}{2}}$ and $|f_1(x)|\leq t^{\frac{1}{2}}$ for all $x\in X\setminus \cup_{i} 6Q_{1,i}$.
\item
$\frac{1}{\mu(6^2 \eta Q_{1,i})}\int_{\eta Q_{1, i}}|f_1|d\mu\lesssim t^{\frac{1}{2}}$ for all $\eta>1$.
\item
$g_1=f_1 \chi_{X\setminus \cup 6Q_{1,i}}+\sum\limits_{i}\phi_{1,i}$, $\beta_{1,i}=\omega_{1,i}f_1-\phi_{1,i}$ with $\beta_{1,i}(Q_{1,i})=0$ and $\|\beta_{1,i}\|_{L^1(X,d,\mu)}\lesssim |f_1|(Q_{1,i})$.
\item
$\|g_1\|_{L^\infty(X,d,\mu)}\lesssim t^{\frac{1}{2}}$, $\|g_1\|_{L^1(X,d,\mu)}\leq 1$ and $\|g_1\|^s_{L^s(X,d,\mu)}\lesssim t^{\frac{s-1}{2}}$ for $1<s<\infty$.
\item
$\supp(\phi_{1,i})\subset R_{1,i}$,~~ $\|\phi_{1,i}\|_{L^\infty(X,d,\mu)}\mu(R_{1,i})\lesssim |f_1|(Q_{1,i})$.
\end{enumerate}
Apply Calder\'on-Zygmund decomposition to $f_2$ also at the same height ($t^{\frac{1}{2}}$) and corresponding notations will be $g_2, Q_{2,j}, R_{2,j}, \beta_{2,j}, \phi_{2,j}, \omega_{2,j}$ etc. As a by product of \CZ~decomposition, we also have the following lemma.
\begin{lem}
\label{revlem}
Let $f_i$ be as above and $\{Q_{i,j}\}_{j}$ are cubes associated to the \CZ~decomposition of $f_i$ at height $t^{\frac{1}{2}}$ for $i=1,2$. Then for $i=1,2$,
$$\sup\limits_{r>0}\frac{|f_i|(B(x,r))}{\lambda(x,r)}\leq \kappa~ t^{\frac{1}{2}},$$
~for all $x\in X\setminus \bigcup_{j} 6^2Q_{i,j}$. The constant $\kappa$ only depends on $C_{\lambda}$ and the doubling dimension of $X$.
\end{lem}
\begin{proof} 
We only prove it for $i=2$. If $B(x,r)\bigcap(\bigcup 6 Q_{2,j})=\emptyset $. Then by \CZ~decomposition (as $|f_2|\leq t^{\frac{1}{2}}$ on $(\bigcup 6 Q_{2,j})^c$) we get,
$$|f_2|(B(x,r))\leq t^{\frac{1}{2}}\mu(B(x,r))\leq C_{\lambda} t^{\frac{1}{2}} \lambda(x,r).$$
If $B(x,r)\cap 6 Q_{2,j}\neq \emptyset$, then as $x\in X\setminus \bigcup 6^2Q_{2,j}$, we have $r>28r(Q_{2,j})$. Thus $B(x,r)\subset B(c_{Q_{2,j}}, Cr)$ and this implies together with \CZ~decomposition
$$|f_2|(B(x,r))\leq |f_2|(B(c_{Q_{2,j}}, Cr))\lesssim t^{\frac{1}{2}}\mu (B(c_{Q_{2,j}}, 6^2 Cr))\leq \kappa t^{\frac{1}{2}} \lambda(x,r),$$
where we have used $\lambda(x,Cr)\simeq \lambda(c_{Q_{2,j}}, Cr)$ as $x\in B(c_{Q_{2,j}}, Cr)$.
\end{proof}
Let us start with the proof of Theorem~\ref{endpoint}.\\

\noindent\textbf{Proof of Theorem~\ref{endpoint}:}
\begin{proof}
We will prove that 
\begin{align}
\mu(\{x\in X: |T_\epsilon(f_1,f_2)(x)|>t\})\lesssim \frac{1}{t^{\frac{1}{2}}}
\label{c6weakmain}
\end{align}
with constant independent of $\epsilon>0$.
\begin{align*}
&\mu(\{x\in X: |T_\epsilon(f_1,f_2)(x)|>t\})\\
&\leq \mu(\bigcup_{i} 6^2Q_{1,i})+\mu(\bigcup_{j}6^{2}Q_{2,j})+\mu(I_{GG})+\mu(I_{BG})+\mu(I_{GB})+\mu(I_{BB}),
\end{align*}
where
\begin{small}
\begin{align*}
&I_{GG}:=\{x\in X:|T_\epsilon(g_1,g_2)(x)|>\frac{t}{4}\},\\
&I_{BG}:=\{x\in X\setminus \bigcup\limits_{i}6^2Q_{1,i}: |T_\epsilon(\beta_1,g_2)(x)|>\frac{t}{4}\},\\
& I_{GB}:=\{x\in X\setminus \bigcup\limits_{j}6^2Q_{2,j}: |T_\epsilon(g_1,\beta_2)(x)|>\frac{t}{4}\},\\
& I_{BB}:=\{x\in X\setminus (\bigcup\limits_{i}6^2Q_{1,i}\cup\bigcup\limits_{j}6^2Q_{2,j}): |T_\epsilon(\beta_1,\beta_2)(x)|>\frac{t}{4}\}.
\end{align*} 
\end{small}
By the decomposition 
\begin{align*}
\mu(\bigcup_{i} 6^2Q_{1,i})\lesssim \frac{1}{t^{\frac{1}{2}}}\sum_{i}\int_{Q_{1,i}}|f_1|\lesssim \frac{1}{t^{\frac{1}{2}}}.
\end{align*}
Similarly $\mu(\bigcup_{j} 6^{2}Q_{2,j})\lesssim \frac{1}{t^{\frac{1}{2}}}$.\\
\textbf{Good-Good:}~Using the boundedness of $T$ we get
\begin{align}
\label{c6gg}
\mu(I_{GG})&\lesssim \frac{1}{t^{p}}\|g_1\|_{L^{p_1}(X,d,\mu)}^p\|g_2\|_{L^{p_2}(X,d,\mu)}^p\lesssim \frac{1}{t^{p}} t^{\frac{p}{2}(\frac{p_{1}-1}{p_1}+\frac{p_{2}-1}{p_{2}})}\lesssim \frac{1}{t^{p}} t^{p-\frac{1}{2}}\lesssim \frac{1}{t^{\frac{1}{2}}}.
\end{align}
\textbf{Bad-Good:} We want to estimate $I_{BG}:=\{x\in X\setminus \bigcup\limits_{i}6^2Q_{1,i}: |T_{\epsilon}(\beta_1,g_2)(x)|>\frac{t}{4}\}$. Observe 
\begin{align}
\nonumber\mu(I_{BG})\lesssim \frac{1}{t}\sum\limits_{i}\int_{X\setminus 2R_{1,i}}|T_{\epsilon}(\beta_{1,i},g_2)(x)|d\mu(x)+\frac{1}{t}\sum\limits_{i}\int_{2R_{1,i}\setminus 6^2 Q_{1,i}}|T_{\epsilon}(\beta_{1,i},g_2)(x)|d\mu(x)
\label{bg1}
\end{align}
Consider the first term. Let $x\notin (2R_{1,i})$ be such that $dist(x, R_{1,i})>\epsilon$, then $\supp(\beta_{1,i})\subset R_{1,i}$. Therefore using the cancellation in the first component we get
\begin{align*}
&\nonumber|T_{\epsilon}(\beta_{1,i},g_2)(x)|\\
\nonumber &=|\int_{d(x,z)>\epsilon}\int_{R_{1,i}}(K(x,y,z)-K(x,y',z))\beta_{1,i}(y)d\mu(y)g_2(z)d\mu(z)|\\
\nonumber &\lesssim \int_{d(x,z)>\epsilon}\int_{R_{1,i}}\omega\left(\frac{d(y,c_{Q_{1,i}})}{(d(x,y)+d(x,z))}\right)\frac{|\beta_{1,i}(y)|d\mu(y)||g_2|d\mu(z)}{(\lambda(x,d(x,y))+\lambda(x,d(x,z)))^2}\\
\nonumber &\lesssim t^{\frac{1}{2}}||\beta_{1,i}||_{1}\omega(\frac{r(R_{1,i})}{d(x,c_{R_{1,i}})})\int_{d(x,z)>\epsilon}\frac{1}{(\lambda(x,d(x,{R_{1,i}}))+\lambda(x,d(x,z)))^2}d\mu(z)\\
\nonumber &\lesssim t^{\frac{1}{2}}||\beta_{1,i}||_{1}\omega(\frac{r(R_{1,i})}{d(x,c_{R_{1,i}})})\left(\int_{\epsilon<d(x,z)\leq d(x, {R_{1,i}})}\frac{1}{\lambda(x,d(x,{R_{1,i}}))^2}d\mu(z)\right)\\
&+ t^{\frac{1}{2}}||\beta_{1,i}||_{1}\omega(\frac{r(R_{1,i})}{d(x,c_{R_{1,i}})})\left(\int_{d(x,z)> d(x, {R_{1,i}})}\frac{1}{\lambda(x,d(x,z))^2}d\mu(z)\right)\\
\nonumber &\lesssim t^{\frac{1}{2}}||\beta_{1,i}||_{1}\omega(\frac{r(R_{1,i})}{d(x,c_{R_{1,i}})})
\frac{1}{\lambda(x,d(x,{R_{1,i}}))}~~~(by~Lemma~\ref{bglemma}.)\\
\end{align*}
Thus
\begin{equation}
|T_{\epsilon}(\beta_{1,i},g_2)(x)|\lesssim t^{\frac{1}{2}}|f_1|({Q_{1,i}})\omega(\frac{r(R_{1,i})}{d(x,c_{R_{1,i}})})
\frac{1}{\lambda(x,d(x,{R_{1,i}}))}.
\label{revised1}
\end{equation}
Now using \eqref{revised1} and standard annular decomposition, we get
\begin{align} 
\int_{X\setminus 2R_{1,i},~dist(x,R_{1,i})>\epsilon}|T_{\epsilon}(\beta_{1,i},g_2)(x)|d\mu(x)
\lesssim C_{\|\omega\|_{Dini(1)}}t^{\frac{1}{2}}|f_1|(Q_{1,i}).
\label{bg3}
\end{align}
Let $x\notin 2R_{1,i}$ with $dist(x, R_{1,i})\leq \epsilon$. Then $d(x,c_{R_{1,i}})\leq C\epsilon$ and using Lemma~\ref{bglemma} and the size-estimate of $K$ we get
\begin{align}
\nonumber |T_{\epsilon}(\beta_{1,i},g_2)(x)|\lesssim t^{\frac{1}{2}}|f_1|(Q_{1,i})\frac{1}{\lambda(x,\epsilon)}.
\end{align}
This implies together with the fact that $\lambda(x,r)\simeq \lambda(y,r)$ if $d(x,y)\leq r$,
\begin{align}
\int_{X\setminus 2R_{1,i},~dist(x,R_{1,i})\leq\epsilon}|T_{\epsilon}(\beta_{1,i},g_2)(x)|d\mu(x)
&\lesssim t^{\frac{1}{2}}|f_1|(Q_{1,i}).
\label{bg4}
\end{align}
Thus \eqref{bg3} and \eqref{bg4} implies 
\begin{align}
\int_{X\setminus 2R_{1,i}}|T_{\epsilon}(\beta_{1,i},g_2)(x)|d\mu(x)\lesssim t^{\frac{1}{2}}|f_1|(Q_{1,i}).
\label{bg5}
\end{align}
This completes the first term. For the second one, 
\begin{align*}&\int_{2R_{1,i}\setminus 6^2 Q_{1,i}}|T_{\epsilon}(\beta_{1,i},g_2)(x)|d\mu(x)\\
&\leq \int_{2R_{1,i}\setminus 6^2 Q_{1,i}}|T_{\epsilon}(w_{1,i}f_1,g_2)(x)|d\mu(x)+\int_{2R_{1,i}}|T_{\epsilon}(\phi_{1,i},g_2)(x)|d\mu(x)\\
&=:T_1+T_2.
\end{align*}

\noindent\textbf{Estimate for $T_1$:} We use the size estimate and the fact that $d(x,y)\simeq d(x, c_{Q_{1,i}})+r(R_{1,i})$ for $x\notin 6^2Q_{1,i}, y\in 6^2Q_{1,i}$. Subsequently applying Lemma~\ref{bglemma} in the first term, we get:
\begin{align*}
&|T_{\epsilon}(w_{1,i}f_1,g_2)(x)|\\
&\leq t^{\frac{1}{2}}|f_1|(Q_{1,i})\textstyle{\left(\int_{d(x,z)>d(x,c_{Q_{1,i}})}+\int_{\epsilon<d(x,z)<d(x,c_{Q_{1,i}})}\right)\frac{d\mu(z)}{(\lambda(x,d(x,c_{Q_{1,i}}))+\lambda(x,d(x,z)))^2}}\\
&\lesssim t^{\frac{1}{2}}|f_1|(Q_{1,i}) \frac{1}{\lambda(x, d(x,c_{Q_{1,i}}))}
\end{align*}

By Lemma~\ref{tolsa}, we get the following:
\begin{align}
T_1 \lesssim t^{\frac{1}{2}}|f_1|(Q_{1,i}).
\label{bg6}
\end{align}
\textbf{Estimate for $T_2$:}
\begin{align*}
T_2 &:=\int_{2R_{1,i}}|T_{\epsilon}(\phi_{1,i},g_2)(x)|d\mu(x)\\
&\leq \int_{2R_{1,i}}|T_{\epsilon}(\phi_{1,i},\chi_{4R_{1,i}}g_2)(x)|d\mu(x)+\int_{2R_{1,i}}|T_{\epsilon}(\phi_{1,i},\chi_{(4R_{1,i})^c}g_2)(x)|d\mu(x)\\
&=:S_a+S_b
\end{align*}
For the first term we use the doubling condition on $R_{1,i}$ and boundedness of $T_{\epsilon}$ to conclude the following:
\begin{align}
\nonumber S_a &\lesssim \mu(R_{1,i})^{1-\frac{1}{p}}\|\phi_{1,i}\|_{L^\infty(X,d,\mu)}\mu(R_{1,i})^{\frac{1}{p_1}}t^{\frac{1}{2}}\mu(R_{1,i})^{\frac{1}{p_2}}\\
&\lesssim t^{\frac{1}{2}}|f_1|(Q_{1,i})
\label{bg7}
\end{align}
For $S_b$, use Lemma~\ref{bglemma} and observe
\begin{align*}
|T_\epsilon(\phi_{1,i},\chi_{(4R_{1,i})^c}g_{2})(x)|
&\leq \int_{(4R_{1,i})^c}\int_{R_{1,i}}\frac{|\phi_{1,i}(y)|d\mu(y)|g_2(z)|d\mu(z)}{(\lambda(x,d(x,y))+\lambda(x,d(x,z)))^2}\\
&\lesssim \|\phi_{1,i}\|_{L^\infty(X,d,\mu)}t^{\frac{1}{2}}\mu(R_{1,i})\frac{1}{\lambda(x, Cr(R_{1,i}))}.
\end{align*}
Now
\begin{align}
\nonumber S_b &\lesssim \|\phi_{1,i}\|_{L^\infty(X,d,\mu)}t^{\frac{1}{2}}\mu(R_{1,i})\int_{2R_{1,i}}\frac{1}{\lambda(x, Cr(R_{1,i}))}d\mu(x)\\
&\lesssim t^{\frac{1}{2}}|f_1|(Q_{1,i}).
\label{bg8}
\end{align}
Thus from \eqref{bg7} and \eqref{bg8}, we get  $T_2\lesssim t^{\frac{1}{2}}|f_1|(Q_{1,i})$. Combining this with \eqref{bg5} and \eqref{bg6}, we get
\begin{align}
\label{BG}
\mu(I_{BG})\lesssim \frac{1}{t}t^{\frac{1}{2}}\sum_{i}|f_1|(Q_{1,i})\lesssim \frac{1}{t^{\frac{1}{2}}}.
\end{align}
The estimate for $I_{GB}$ follows exactly from the arguments as in $I_{BG}$. Thus we skip it.\\
\textbf{Bad-Bad:} $BB:=\{x\in X\setminus \mathcal{B}: |T_\epsilon(\beta_1,\beta_2)(x)|>\frac{t}{4}\}$ where $\mathcal{B}=\bigcup\limits_{i}6^2Q_{1,i}\cup\bigcup\limits_{j}6^2 Q_{2,j}$.
$$T_\epsilon(\beta_1,\beta_2)=\sum\limits_{i}T_\epsilon\left(\beta_{1,i}, \sum\limits_{j:r(R_{1,i})\leq r(R_{2,j})}\beta_{2,j}\right)+\sum_{j}T_\epsilon\left(\sum\limits_{i:r(R_{1,i})>r(R_{2,j}))}\beta_{1,i}, \beta_{2,j}\right).$$
The above two terms are symmetric and thus we only treat the first. Denote $\mathcal{T}_i:=\{j: r(R_{1,i})\leq r(R_{2,j})\}$. 
\begin{small}
\begin{align}
\nonumber &\mu(\{x\in X\setminus\mathcal{B}:\sum\limits_{i}|T_\epsilon(\beta_{1,i}, \sum\limits_{j\in \mathcal{T}_i}\beta_{2,j})(x)|>\frac{t}{8}\})\\
\nonumber &\leq \mu(\{x\in X\setminus\mathcal{B}:\sum\limits_{i}\chi_{(2R_{1,i})^c}|T_\epsilon(\beta_{1,i}, \sum\limits_{j\in \mathcal{T}_i}\beta_{2,j})(x)|>\frac{t}{16}\})\\
\nonumber &+\mu(\{x\in X\setminus\mathcal{B}:\sum\limits_{i}\chi_{2R_{1,i}}|T_\epsilon(\beta_{1,i}, \sum\limits_{j\in \mathcal{T}_i}\beta_{2,j})(x)|>\frac{t}{16}\})\\
&=:\mathcal{S}_1+\mathcal{S}_2.
\label{bb1}
\end{align}
\end{small}
We only give the estimate for $\mathcal{S}_1$. Estimates for $\mathcal{S}_2$ follows from a similar line of arguments and incorporating the ideas from \cite{DLi,marti}.\vspace{1em}

\noindent\textbf{Estimate for $\mathcal{S}_1$:}\\
$$\mathcal{S}_1 \lesssim I_{a1}+I_{a2}+I_{a3}+I_{a4}+I_b,$$
where,
\begin{align}
\nonumber & I_{a1}:=\frac{1}{t}\int_{X\setminus\mathcal{B}}\left(\sum_{i,j}\chi_{(2R_{1,i})^c}\chi_{(2R_{2,j})^c}\chi_{\min\{d(x, R_{1,i}),d(x, R_{2,j})\}>\epsilon}|T_\epsilon(\beta_{1,i},\beta_{2,j})(x)|\right)d\mu(x),\\
\nonumber & I_{a2}:=\frac{1}{t^{\frac{1}{2}}}\int_{X\setminus\mathcal{B}}\left(\sum_{i,j}\chi_{(2R_{1,i})^c}\chi_{(2R_{2,j})^c}\chi_{\{d(x, R_{1,i})>\epsilon,~d(x, R_{2,j})\leq\epsilon\}}|T_\epsilon(\beta_{1,i},\beta_{2,j})(x)|\right)^{\frac{1}{2}}d\mu(x),\\
\nonumber & I_{a3}:=\frac{1}{t^{\frac{1}{2}}}\int_{X\setminus\mathcal{B}}\left(\sum_{i,j}\chi_{(2R_{1,i})^c}\chi_{(2R_{2,j})^c}\chi_{\{d(x, R_{2,j})>\epsilon,~d(x, R_{1,i})\leq\epsilon\}}|T_\epsilon(\beta_{1,i},\beta_{2,j})(x)|\right)^{\frac{1}{2}}d\mu(x),\\
\nonumber & I_{a4}:=\frac{1}{t^{\frac{1}{2}}}\int_{X\setminus\mathcal{B}}\left(\sum_{i,j}\chi_{(2R_{1,i})^c}\chi_{(2R_{2,j})^c}\chi_{\{\max\{d(x, R_{1,i}),d(x, R_{2,j})\}\leq\epsilon\}}|T_\epsilon(\beta_{1,i},\beta_{2,j})(x)|\right)^{\frac{1}{2}}d\mu(x),\\
\nonumber & I_{b}:=\frac{1}{t^{\frac{1}{2}}}\int_{X\setminus\mathcal{B}}\left(\sum_{i,j}\chi_{(2R_{1,i})^c}\chi_{(2R_{2,j})}|T_\epsilon(\beta_{1,i},\beta_{2,j})(x)|\right)^{\frac{1}{2}}d\mu(x).
\end{align}

\noindent\textbf{Estimate for $I_{a1}$:} \\
Using cancellation at the first component (for $\beta_{1,i}$), $|T_{\epsilon}(\beta_{1,i}, \beta_{2,j})(x)|$ is dominated by
\begin{align*}
&||\beta_{1,i}||_{1}||\beta_{2,j}||_{1}\omega\left(\frac{r(R_{1,i})}{(d(x,c_{R_{1,i}})+d(x,c_{R_{2,j}}))}\right)\frac{1}{(\lambda(x,d(x,c_{R_{1,i}}))+\lambda(x,d(x,c_{R_{2,j}})))^2}\\
\lesssim &||\beta_{1,i}||_{1}\int \chi_{Q_{2,j}}\omega\left(\frac{r(R_{1,i})}{(d(x,c_{R_{1,i}})+d(x,c_{R_{2,j}}))}\right)\frac{|f_2(z)|d\mu(z)}{(\lambda(x,d(x,c_{R_{1,i}}))+\lambda(x,d(x,z)))^2}.
\end{align*}
Now we sum over all $j$ and as $\sum_{j}\chi_{Q_{2,j}}\leq C$, the sum is dominated by the following
\begin{align*}
||\beta_{1,i}||_{1}\int_{d(x,z)>\epsilon}\omega\left(\frac{r(R_{1,i})}{(d(x,c_{R_{1,i}})+d(x,c_{R_{2,j}}))}\right)\frac{|f_2(z)|d\mu(z)}{(\lambda(x,d(x,c_{R_{1,i}}))+\lambda(x,d(x,z)))^2}
\end{align*}
Now we split the last quantity as following:
\begin{align*}
& ||\beta_{1,i}||_{1}\int_{\epsilon<d(x,z)\leq d(x, c_{R_{1,i}})}\omega\left(\frac{r(R_{1,i})}{(d(x,c_{R_{1,i}})+d(x,c_{R_{2,j}}))}\right)\frac{|f_2(z)|d\mu(z)}{(\lambda(x,d(x,c_{R_{1,i}}))+\lambda(x,d(x,z)))^2}\\
& +||\beta_{1,i}||_{1}\int_{d(x,z)>d(x, c_{R_{1,i}})}\omega\left(\frac{r(R_{1,i})}{(d(x,c_{R_{1,i}})+d(x,c_{R_{2,j}}))}\right)\frac{|f_2(z)|d\mu(z)}{(\lambda(x,d(x,c_{R_{1,i}}))+\lambda(x,d(x,z)))^2}.
\end{align*}
Now using Lemma~\ref{revlem}, the first term is simply dominated by $$\kappa||\beta_{1,i}||_{1}\omega\left(\frac{r(R_{1,i})}{d(x,c_{R_{1,i}})}\right)\frac{t^{\frac{1}{2}}}{\lambda(x,d(x,c_{R_{1,i}}))}.$$
Now for the second term construct a sequence of radii as in the Lemma~\ref{bglemma} and call them $\{r_l\}$. Denote $r_0=d(x, c_{R_{1,i}})$ and let $k_l$ be the smallest positive integer corresponding to $r_l$, i.e, $r_{l+1}=2^{k_l}r_l$ such that $\lambda(x, 2^{k_l}r_{l})>2\lambda(x, r_l)$, and $\lambda(x, 2^{k_{l}-1}r_{l})\leq 2\lambda(x, r_l)$, then the integral is dominated by
\begin{align*}
&||\beta_{1,i}||_{1}\omega(\frac{r(R_{1,i})}{d(x, c_{R_{1,i}})})\sum_{l}\int_{B(x,r_{l+1})\setminus B(x, r_l)}\frac{|f_2(z)|}{\lambda(x,d(x,z))^2}d\mu(z)\\
& \leq ||\beta_{1,i}||_{1}\omega(\frac{r(R_{1,i})}{d(x, c_{R_{1,i}})})\sum_{l}\frac{|f_2|(B(x, r_{l+1}))}{\lambda(x,r_{l+1})^2}\\
&\lesssim \kappa||\beta_{1,i}||_{1}\omega(\frac{r(R_{1,i})}{d(x, c_{R_{1,i}})})\sum_{l}\frac{t^{\frac{1}{2}}\lambda(x, r_{l+1})}{\lambda(x,r_{l+1})^2}~~(by~Lemma~\ref{revlem})\\
&\lesssim ||\beta_{1,i}||_{1}\omega(\frac{r(R_{1,i})}{d(x, c_{R_{1,i}})})\frac{t^{\frac{1}{2}}}{\lambda(x, r_0)}\sum_{l}\frac{1}{2^l}
\lesssim ||\beta_{1,i}||_{1}\omega\left(\frac{r(R_{1,i})}{d(x,c_{R_{1,i}})}\right)\frac{t^{\frac{1}{2}}}{\lambda(x,d(x,c_{R_{1,i}}))}.
\end{align*}
Now first integrating over $X\setminus 2R_{1,i}$ and then summing over $i$, gives that
\begin{equation}
\label{revisedbb1}
I_{a1}\lesssim \frac{\|\omega\|_{Dini(1)}}{t^{\frac{1}{2}}}.
\end{equation}

\noindent\textbf{Estimate for $I_{a2}$:} \\
For $d(., R_{1,i})>\epsilon$, $R_{1,i}\subset B(x,\epsilon)^c$, we can use the cancellation in the first component to obtain the following:
\begin{align*}
|T_\epsilon(\beta_{1,i}, \beta_{2,j})(x)|
\lesssim \omega\left(\frac{r(R_{1,i})}{d(x,c_{R_{1,i}})}\right)\frac{|f_1|(Q_{1,i})}{\lambda(x,d(x,c_{R_{1,i}}))}\frac{|f_2|(Q_{2,j})}{\lambda(x,\epsilon)}
\end{align*}
As $d(x, R_{2,j})\leq \epsilon$, we have $d(x,c_{R_{2,j}})\leq C\epsilon$. Thus using the above estimate with H\"older's inequality and subsequently using the annular decomposition we get
\begin{align}
\nonumber I_{a2}
&\lesssim \frac{1}{t^{\frac{1}{2}}} \left(\sum_{i}\sum_{k}\int_{2^{k+1}r(R_{1,i}) \geq d(x,c_{R_{1,i}})>2^k r(R_{1,i})}\omega\left(\frac{r(R_{1,i})}{d(x,c_{R_{1,i}})}\right)\frac{|f_1|(Q_{1,i})}{\lambda(x,d(x,c_{R_{1,i}}))}\right)^{\frac{1}{2}}\\
\nonumber &\left(\sum_{j}\int_{B(c_{R_{2,j}}, C\epsilon)}\frac{|f_2|(Q_{2,j})}{\lambda(x,\epsilon)}d\mu(x)\right)^{\frac{1}{2}}\\
&\lesssim \frac{\|\omega\|^{\frac{1}{2}}_{Dini(1)}}{t^{\frac{1}{2}}}.
\label{bb5}
\end{align}

\noindent\textbf{Estimate for $I_{a3}$:} Similar to $I_{a2}$.\vspace{1em}

\noindent\textbf{Estimate for $I_{a4}$:}\\
When $x\in (2R_{1,i})^c\cap(2R_{2,j})^c $ with $d(x, R_{1,i})\leq \epsilon, d(x, R_{2,j})\leq\epsilon$. We have,
$$|T_\epsilon(\beta_{1,i},\beta_{2,j})(x)|\lesssim \frac{|f_1|(Q_{1,i})}{\lambda(x,\epsilon)}\frac{|f_2|(Q_{2,j})}{\lambda(x,\epsilon)}.$$
This implies 
\begin{align}
\nonumber I_{a4} &\lesssim \frac{1}{t^{\frac{1}{2}}}\left(\sum_{i}\int_{B(c_{R_{1,i}}, C\epsilon)}\frac{|f_1|(Q_{1,i})}{\lambda(x,\epsilon)}d\mu(x)\right)^{\frac{1}{2}}\left(\sum_{j}\int_{B(c_{R_{2,j}}, C\epsilon)}\frac{|f_2|(Q_{2,j})}{\lambda(x,\epsilon)}d\mu(x)\right)^{\frac{1}{2}}\\
&\lesssim \frac{1}{t^{\frac{1}{2}}}\left(\sum_{i}|f_1|(Q_{1,i})\right)^{\frac{1}{2}}\left(\sum_{j}|f_2|(Q_{2,j})\right)^{\frac{1}{2}}\lesssim \frac{1}{t^{\frac{1}{2}}}.
\label{bb7}
\end{align}

\noindent\textbf{Estimate for $I_b$:}\\
\begin{align}
\nonumber I_b \nonumber &\leq \frac{1}{t^{\frac{1}{2}}}\int_{X\setminus\mathcal{B}}\left(\sum_{i,j}\chi_{(2R_{1,i})^c}\chi_{(2R_{2,j})}|T_\epsilon(\beta_{1,i},\phi_{2,j})(x)|\right)^{\frac{1}{2}}d\mu(x)\\
\nonumber &+\frac{1}{t^{\frac{1}{2}}}\int_{X\setminus\mathcal{B}}\left(\sum_{i,j}\chi_{(2R_{1,i})^c}\chi_{(2R_{2,j})\setminus6^2Q_{2,j}}|T_\epsilon(\beta_{1,i},w_{2,j}f_2)(x)|\right)^{\frac{1}{2}}d\mu(x)\\
&:=I_b'+I_b''.
\label{bb9}
\end{align}
\textbf{Estimate for $I_b'$:}\\
\textbf{Case~1:}
For $I_b'$, consider $x\in (2R_{1,i})^c\cap 2R_{2,j}$ such that $d(x, R_{1,i})>\epsilon$. Then using the cancellation in the first component, similar to \eqref{revised1}, we get
$$|T_\epsilon(\beta_{1,i},\phi_{2,j})(x))|\lesssim |f_1|(Q_{1,i})\|\phi_{2,j}\|_{L^\infty(X,d,\mu)}\frac{1}{\lambda(x,d(x, R_{1,i}))}\omega\left(\frac{r(R_{1,i})}{d(x,c_{R_{1,i}})}\right).$$
Now using H\"older's inequality and the above estimate we get,
\begin{align}
\nonumber &\frac{1}{t^{\frac{1}{2}}}\int_{X\setminus\mathcal{B}}\left(\sum_{i,j}\chi_{(2R_{1,i})^c}\chi_{d(.,R_{1,i})>\epsilon}\chi_{(2R_{2,j})}|T_\epsilon(\beta_{1,i},\phi_{2,j})(x)|\right)^{\frac{1}{2}}d\mu(x)\\
\nonumber &\lesssim \frac{1}{t^{\frac{1}{2}}}\left(\sum_{i}|f_1|(Q_{1,i})\sum_{k}\int\limits_{2^kr(R_{1,i})\leq d(x, c_{R_{1,i}})<2^{k+1}r(R_{1,i})}\frac{1}{\lambda(x,d(x, R_{1,i}))}\omega\left(\frac{r(R_{1,i})}{d(x,c_{R_{1,i}})}\right)\right)^{\frac{1}{2}}\\
 \nonumber & \times \left(\sum_{j}\|\phi_{2,j}\|_{L^\infty(X,d,\mu)}\mu(R_{2,j})\right)^{\frac{1}{2}}\lesssim \left(\sum_{i}|f_1|(Q_{1,i})\sum_{k}\omega(2^{-k})\right)^\frac{1}{2}\left(\sum_{j}|f_2|(Q_{2,j})\right)^{\frac{1}{2}}\\
&\lesssim \frac{\|\omega\|^{\frac{1}{2}}_{Dini(1)}}{t^{\frac{1}{2}}}.
\label{bb10}
\end{align}

\textbf{Case~2:}
Consider $x\in (2R_{1,i})^c\cap 2R_{2,j}$ with $d(x, R_{1,i})\leq \epsilon$. Then the size condition with Lemma~\ref{bglemma} gives the following estimate:
\begin{align*}
&|T_\epsilon(\beta_{1,i},\phi_{2,j})(x)|\\
&\lesssim |f_1|(Q_{1,i})\|\phi_{2,j}\|_{L^\infty(X,d,\mu)}\left(\int_{d(x,z)>\epsilon}\frac{d\mu(z)}{\lambda(x,d(x,z))^2}+\int_{d(x,z)\leq \epsilon, d(x,y)>\epsilon}\frac{1}{\lambda(x,\epsilon)^2}d\mu(z)\right)\\
&\lesssim |f_1|(Q_{1,i})\|\phi_{2,j}\|_{L^\infty(X,d,\mu)} \frac{1}{\lambda(x,\epsilon)}.
\end{align*}
Now using the above estimate we get
\begin{align}
\nonumber &\frac{1}{t^{\frac{1}{2}}}\int_{X\setminus\mathcal{B}}\left(\sum_{i,j}\chi_{(2R_{1,i})^c}\chi_{(2R_{2,j})}\chi_{d(.,R_{1,i})\leq \epsilon}|T_\epsilon(\beta_{1,i},\phi_{2,j})(x)|\right)^{\frac{1}{2}}d\mu(x)\\
\nonumber &\lesssim \frac{1}{t^{\frac{1}{2}}}\left(\int_{X}\sum_{i}\chi_{B(c_{R_{1,i}}, C\epsilon)}\frac{|f_1|(Q_{1,i})}{\lambda(x,\epsilon)}\right)^{\frac{1}{2}}\left(\int_{X}\sum_{j}\chi_{2R_{2,j}}\|\phi_{2,j}\|_{L^\infty(X,d,\mu)}\right)^{\frac{1}{2}}\\
&\lesssim \frac{1}{t^{\frac{1}{2}}}(\sum_i |f_1|(Q_{1,i}))^{\frac{1}{2}}(\sum_j |f_2|(Q_{2,j}))^{\frac{1}{2}}\lesssim \frac{1}{t^{\frac{1}{2}}}.
\label{bb11}
\end{align}
Combining \eqref{bb10},\eqref{bb11} we get
\begin{align}
I_b'\lesssim \frac{1}{t^{\frac{1}{2}}}.
\label{bb12}
\end{align}
\vspace{1em}
\noindent\textbf{Estimate for $I_b''$}:\\
\textbf{Case~1:} Consider $x$ such that $d(x, R_{1,i})>\epsilon$. Then using cancellation at the first component we get
$$|T_\epsilon(\beta_{1,i}, w_{2,j}f_2)(x)|\lesssim \omega\left(\frac{r(R_{1,i})}{d(x,c_{R_{1,i}})}\right)\frac{|f_1|(Q_{1,i})}{\lambda(x,d(x,c_{R_{1,i}}))}\frac{|f_2|(Q_{2,j})}{\lambda(x, d(x,c_{R_{2,j}}))}.$$
Now using Lemma~\ref{tolsa} at the last but one step we get,
\begin{align}
\nonumber &\frac{1}{t^{\frac{1}{2}}}\int_{X\setminus\mathcal{B}}\left(\sum_{i,j}\chi_{(2R_{1,i})^c}\chi_{(2R_{2,j})\setminus6^2Q_{2,j}}\chi_{d(.,R_{1,i})>\epsilon}|T_\epsilon(\beta_{1,i},w_{2,j}f_2)(x)|\right)^{\frac{1}{2}}d\mu(x)\\
\nonumber &\lesssim\frac{1}{t^{\frac{1}{2}}}\left(\int_{X}\sum_{i}\chi_{(2R_{1,i})^c}\omega\left(\frac{r(R_{1,i})}{d(x,c_{R_{1,i}})}\right)\frac{|f_1|(Q_{1,i})}{\lambda(x,d(x,c_{R_{1,i}}))}d\mu(x)\right)^{\frac{1}{2}}\\
\nonumber &\times\left(\int_{X}\sum_j \chi_{(2R_{2,j})\setminus6^2Q_{2,j}} \frac{|f_2|(Q_{2,j})}{\lambda(x, d(x,c_{R_{2,j}}))} d\mu(x) \right)^{\frac{1}{2}}\\
\nonumber &\lesssim \frac{\|\omega\|^{\frac{1}{2}}_{Dini(1)}}{t^{\frac{1}{2}}}(\sum_i |f_1|(Q_{1,i}))^{\frac{1}{2}}\left(\sum_j \int_{2R_{2,j}\setminus6^2Q_{2,j}} \frac{|f_2|(Q_{2,j})}{\lambda(x, d(x,c_{R_{2,j}}))} d\mu(x) \right)^{\frac{1}{2}}\\
&\lesssim \frac{\|\omega\|^{\frac{1}{2}}_{Dini(1)}}{t^{\frac{1}{2}}}.
\label{bb13}
\end{align}
\textbf{Case~2:}
For $x\in (2R_{1,i})^c\cap (6^2 Q_{2,j})^c$ with $d(x, R_{1,i})\leq \epsilon$. Then
\begin{align*}
|T_\epsilon(\beta_{1,i},w_{2,j}f_2)(x)|&\lesssim \int_{R_{2,j}}\int_{R_{1,i}}\frac{1}{(\lambda(x,d(x,y))+\lambda(x,d(x,z)))^2}|\beta_{1,i}(y)||w_{2,j}f_2|d\mu(y)d\mu(z)\\
&\lesssim \frac{|f_1|(Q_{1,i})}{\lambda(x,\epsilon)}\frac{|f_2|(Q_{2,j})}{\lambda(x,d(x, c_{R_{2,j}}))}.
\end{align*}
Now using the above estimate and Lemma~\ref{tolsa},
\begin{align}
\nonumber & \frac{1}{t^{\frac{1}{2}}}\int_{X\setminus\mathcal{B}}\left(\sum_{i,j}\chi_{(2R_{1,i})^c}\chi_{(2R_{2,j})\setminus6^2Q_{2,j}}\chi_{d(.,R_{1,i})\leq\epsilon}|T_\epsilon(\beta_{1,i},w_{2,j}f_2)(x)|\right)^{\frac{1}{2}}d\mu(x)\\
\nonumber &\lesssim \frac{1}{t^{\frac{1}{2}}}\left(\sum_i |f_1|(Q_{1,i})\int_{B(c_{R_{1,i}}, C\epsilon)}\frac{d\mu(x)}{\lambda(x,\epsilon)}\right)^{\frac{1}{2}}\left(\sum_j |f_2|(Q_{2,j}) \int_{(2R_{2,j})\setminus6^2Q_{2,j}}\frac{d\mu(x)}{\lambda(x,d(x, c_{R_{2,j}}))}\right)^{\frac{1}{2}}\\
&\lesssim \frac{1}{t^{\frac{1}{2}}}.
\label{bb14} 
\end{align}

This completes the estimate for $I_b''$. We get from \eqref{bb13} and \eqref{bb14}
\begin{align}
I_b''\lesssim \frac{1}{t^{\frac{1}{2}}}.
\label{bb15}
\end{align}
Thus
\begin{align}
I_b\lesssim \frac{1}{t^{\frac{1}{2}}}.
\label{bb16}
\end{align}
From \eqref{revisedbb1},~\eqref{bb5},~\eqref{bb7},~\eqref{bb16} we get
\begin{align}
\label{bb17}
\mathcal{S}_{1}\lesssim \frac{C_{T, \|\omega\|_{Dini(1)}}}{t^{\frac{1}{2}}}.
\end{align}
Arguing along the ideas developed in \cite{DLi,marti} and providing necessary modifications (as we have given for $\mathcal{S}_1$) we can prove
\begin{align}
\mathcal{S}_{2}\lesssim \frac{C_{T, \|\omega\|_{Dini(1)}}}{t^{\frac{1}{2}}}.
\label{bb31}
\end{align}
From \eqref{bb1}, \eqref{bb17} and \eqref{bb31} we get
\begin{align}
\label{bb32}
&\mu(\{x\in X\setminus\mathcal{B}:\sum\limits_{i}|T_\epsilon(\beta_{1,i}, \sum\limits_{j\in \mathcal{T}_i}\beta_{2,j})(x)|>\frac{t}{8}\})\lesssim \frac{C_{T, \|\omega\|_{Dini(1)}}}{t^{\frac{1}{2}}}.
\end{align}
These estimates imply $\mu(I_{BB})\lesssim \frac{C_{T, \|\omega\|_{Dini(1)}}}{t^{\frac{1}{2}}}$. This completes the proof.
\end{proof}

\subsection{Cotlar's lemma}
\label{CL}
In this subsection, end-point boundedness for the multilinear maximal Calder\'on-Zygmund operator will be addressed. We will conclude this with a Cotlar's inequality for multi-linear Calder\'on-Zygmund operators on non-homogeneous spaces. Cotlar's inequality for Calder\'on-Zygmund operators for polynomially growth measures was first obtained in \cite{NTV2}. The proof substantially depends on the Guy-David lemma mentioned in \cite{NTV2}. In \cite{Hlyy12}, the authors extended this to the general framework of non-homogeneous spaces. The proof the Theorem~\ref{T*} requires the following lemma which enables us to pass from arbitrary balls to ``doubling ball". The idea of the proof is inspired by~Lemma~3.1 in \cite{Hlyy12}.
\begin{lem}
\label{double}
Let $T$ be a bilinear Calder\'{o}n-Zygmund operator defined on an  upper doubling, geometrically doubling metric measure space $(X,d,\mu)$. Then for any $x\in supp(\mu),~r>0$, there is an integer $k=k(x,r)\in \mathbb{N}$ such that 
\begin{align}
\label{c6compare}
|\tilde{T}_r(f_1,f_2)(x)-\tilde{T}_{5R}(f_1,f_2)(x)|\lesssim_{\lambda,\mu}\mathscr{M}f_{1}(x)\mathscr{M}f_{2}(x)
\end{align}
where $R=5^{k-1}r$ and the ball $B(x,R)$ is doubling i.e., $\mu(B(x,25R))\lesssim \mu(B(x,R))$.
\end{lem}
\begin{proof}
Let $x\in supp(\mu)$ and $r>0$. We claim that there is some $j\in \mathbb{N}$ such that $\mu(B(x, 5^{j+1}r))\leq 4C_{\lambda}^{6}\mu(B(x, 5^{j-1}r))$. If the claim is not true then using the upper doubling condition we get
$$\mu(B(x,r))\leq (4C_{\lambda}^6)^{-j}\mu(B(x, 5^{2j}r ))\lesssim (4C_{\lambda}^6)^{-j}\lambda(x, 5^{2j}r)\leq 5^{-j}\lambda(x,r).$$
Letting $j\rightarrow \infty$ we get $\mu(B,r)=0$ but this is a contradiction as $\mu(B,r)>0$ for $x\in supp(\mu)$. Thus the claim is true.
 
Let $k$ be the smallest integer such that the claim happens thus $\mu(B(x,25R))\lesssim \mu(B(x,R))$ where $R=r_{k-1}=5^{k-1}r$. Denote $~~~ \mu_j=\mu(B(x,5^jr)) ~~~\text{for}~~j=1,\dots,k.$,  then we have 
$
\mu_{j+1}\leq (2C_{\lambda}^3)^{j+2-k}\mu_{k}$ and $\lambda(x,r_k)\leq (C_{\lambda}^3)^{k-j-1}\lambda(x,r_{j+1})~\text{for}~j=1,\dots,k.
$
Now using the above estimates we obtain
\begin{small}
\begin{align}
&\nonumber|\tilde{T}_r(f_1,f_2)(x)-\tilde{T}_{5R}(f_1,f_2)(x)|\\
&\nonumber\leq \sum\limits_{j=1}^{k-1}\int_{B(x,r_{j+1})\setminus B(x,r_{j})}\frac{|f_1(y_1)|}{\lambda(x,d(x,y_1))}d\mu(y_1)\sum\limits_{l=1}^{k-1}\int_{B(x,r_{l+1})\setminus B(x,r_{l})}\frac{|f_2(y_2)|}{\lambda(x,d(x,y_2))}d\mu(y_2)\\
\nonumber&+\frac{1}{\lambda(x,r)}\int_{B(x,r)}|f_1(y_1)|d\mu(y_1)\sum\limits_{l=1}^{k-1}\int_{B(x,r_{l+1})\setminus B(x,r_{l})}\frac{|f_2(y_2)|}{\lambda(x,d(x,y_2))}d\mu(y_2)\\
\nonumber&+\sum\limits_{j=1}^{k-1}\int_{B(x,r_{j+1})\setminus B(x,r_{j})}\frac{|f_1(y_1)|}{\lambda(x,d(x,y_1))}d\mu(y_1)\frac{1}{\lambda(x,r)}\int_{B(x,r)}|f_2(y_2)|d\mu(y_2)\\
\nonumber&\lesssim \sum\limits_{j=1}^{k-1}2^{j-k}\mathscr{M}f_1(x)\sum\limits_{l=1}^{k-1}2^{l-k}\mathscr{M}f_2(x)+\mathscr{M}f_1(x)\sum\limits_{l=1}^{k-1}2^{l-k}\mathscr{M}f_2(x)\\
\nonumber&+\mathscr{M}f_2(x)\sum\limits_{j=1}^{k-1}2^{j-k}\mathscr{M}f_1(x)\\
&\lesssim \mathscr{M}f_1(x)\mathscr{M}f_2(x).
\end{align}
\end{small}
\end{proof}

\noindent\textbf{Proof of Theorem~\ref{T*}:}\\
For multilinear Calder\'on-Zygmund operators L. Grafakos and R. Torres obtained the Cotlar's inequality in \cite{GT1}. We combine their ideas with Lemma~\ref{double} to obtain Theorem~\ref{T*}. Let $x\in X$ and $r>0$, then by Lemma~\ref{double}
\begin{align}
\label{c6use1}
|\tilde{T}_r(f_1,f_2)(x)-\tilde{T}_{5R}(f_1,f_2)(x)|\lesssim_{\lambda,\mu}\mathscr{M}f_{1}(x)\mathscr{M}f_{2}(x)
\end{align}
where $R=5^{k-1}r$ and $\mu(B(x,25R))\lesssim \mu(B(x,R))$. For any $z\in B(x,R)$
\begin{equation}
\label{c6use2}
\tilde{T}_{5R}(f_1,f_2)(z)=T(f_1,f_2)(z)-T(f_{1}^0,f_2^0)(z)
\end{equation}
where $f_{i}^0(x)=f_{i}\chi_{B(x,5R)}$ for $i=1,2$. Denote $S_r(x)=\lbrace\vec{y}\in X^{2}:\max\limits_{i=1,2} d(x,y_i)\leq r\rbrace$ and by the regularity condition we get
\begin{small}
\begin{align*}
&|\tilde{T}_{5R}(f_1,f_2)(x)-\tilde{T}_{5R}(f_1,f_2)(z)|\\
&\lesssim \int_{S_{5R}(x)^c}\left(\min\limits_{i=1,2}\frac{1}{\lambda(x,d(x,y_i))^2}\right)\omega(\frac{d(x,z)}{d(x,y_1)+d(x,y_2)})\prod\limits_{i=1}^{2}|f_i(y_i)|d\vec{\mu}\\
&\lesssim I+II+III
\end{align*}
\begin{align*}
\mbox{where}\\
&I=\int_{d(x,y_1)\leq 5R}\int_{d(x,y_2)>5R}\left(\min\limits_{i=1,2}\frac{1}{\lambda(x,d(x,y_i))^2}\right)\omega(\frac{d(x,z)}{d(x,y_1)+d(x,y_2)})\prod\limits_{i=1}^{2}|f_i(y_i)|d\vec{\mu}\\
&\leq M_{\lambda}(f_1)(x)\sum\limits_{l\geq 0}\omega(2^{-l})\frac{1}{\lambda(x,5R2^{l+1})}\int_{d(x,y_2)\leq 5R2^{l+1}}|f_2(y_2)|d\mu(y_2)\\
&\lesssim M_{\lambda}(f_1)(x)M_{\lambda}(f_2)(x)\\
\mbox{and}\\
&II=\int_{d(x,y_2)\leq 5R}\int_{d(x,y_1)>5R}\left(\min\limits_{i=1,2}\frac{1}{\lambda(x,d(x,y_i))^2}\right)\omega(\frac{d(x,z)}{d(x,y_1)+d(x,y_2)})\prod\limits_{i=1}^{2}|f_i(y_i)|d\vec{\mu}
\end{align*}
Arguing as in $I$ we obtain $II\lesssim M_{\lambda}(f_1)(x)M_{\lambda}(f_2)(x).$
For $III$ we observe
\begin{align*}
&III \\
&=\int_{d(x,y_1)>5R}\int_{d(x,y_2)>5R}\left(\min\limits_{i=1,2}\frac{1}{\lambda(x,d(x,y_i))^2}\right)\omega(\frac{d(x,z)}{d(x,y_1)+d(x,y_2)})\prod\limits_{i=1}^{2}|f_i(y_i)|d\vec{\mu}\\
& \leq \sum\limits_{j\geq 0}\int\limits_{d(x,y_2)\leq 2^{j+1}5R}|f_2(y_1)|d\mu(y_2)\frac{1}{\lambda(x,2^{j+1}5R)^2}\int\limits_{d(x,y_1)\leq 2^{j+1}5R}|f_1(y_1)|d\mu(y_1)\omega(2^{-j})\\
&+\sum\limits_{l\geq 0}\sum\limits_{j\geq 1}^{l}\int\limits_{2^j 5R<d(x,y_1)\leq 2^{j+1}5R}|f_1(y_1)|d\mu(y_1)\frac{1}{\lambda(x,2^{l+1}5R)^2}\int\limits_{d(x,y_2)\leq 2^{l+1}5R}|f_{2}|d\mu(y_2)\omega(2^{-l})\\
&\lesssim M_{\lambda}(f_1)(x)M_{\lambda}(f_2)(x).
\end{align*}
\end{small}
Hence 
\begin{equation}
\label{c6use3}
|\tilde{T}_{5R}(f_1,f_2)(x)-\tilde{T}_{5R}(f_1,f_2)(z)|\lesssim M_{\lambda}(f_1)(x)M_{\lambda}(f_2)(x).
\end{equation}
Now from \eqref{c6use2} and \eqref{c6use3} we obtain
\begin{equation}
\label{c6use4}
|\tilde{T}_{5R}(f_1,f_2)(x)|\lesssim M_{\lambda}(f_1)(x)M_{\lambda}(f_2)(x)+|T(f_1,f_2)(z)|+|T(f_{1}^0,f_2^0)(z)|
\end{equation}
for all $z\in B(x,R)$. Fix $0<\eta<\frac{1}{2}$. Raising \eqref{c6use4} to the power $\eta$ and averaging over $B=B(x,R)$ we get
\begin{small}
\begin{align}
\label{c6use5}
\nonumber|\tilde{T}_{5R}(f_1,f_2)(x)|^{\eta}\lesssim &M_{\lambda}(f_1)(x)M_{\lambda}(f_2)(x)+\mathscr{M}(|T(f_1,f_2)|^{\eta})(x)\\
&+\frac{1}{\mu(B(x,R))}\int_{B(x,R)}|T(f_{1}^0,f_2^0)(z)|^{\eta}d\mu(z).
\end{align}
Now arguing as in \cite{GT1} and using the facts that $T$ maps $L^1(X,d,\mu)\times L^1(X,d,\mu)$ to $L^{\frac{1}{2}, \infty}(X,d,\mu)$ (Theorem~\ref{endpoint}) and $\mu(B(x,25R))\lesssim \mu(B(x,R))$ we obtain
\begin{align}
\nonumber\frac{1}{\mu(B(x,R))}\int_{B(x,R)}|T(f_{1}^0,f_2^0)(z)|^{\eta}d\mu(z)&\leq C_{\eta,T}\mu(B(x,R))^{-2\eta}(\prod_{i=1}^{2}\|f_{i}^0\|_{L^1(X,d, \mu)})^{\eta}\\
\nonumber &\leq C_{\eta,T}\left(\prod_{i=1}^{2}\frac{1}{\mu(B(x,25R))}\int_{B(x,5R)}|f_i|d\mu(y_i)\right)^{\eta}\\
&\leq C_{\eta,T}(\mathscr{M}(f_1)(x)\mathscr{M}(f_2)(x))^{\eta}
\label{c6use6}
\end{align}
Now combining the estimates \eqref{c6use1}, \eqref{c6use4}, \eqref{c6use5} and \eqref{c6use6}, we obtain 
\begin{align}
\label{usefinal}
\nonumber|\tilde{T}_r(f_1,f_2)(x)|\leq C_{\eta,\lambda,T}M_{\lambda}(f_1)(x)M_{\lambda}(f_2)(x)&+\mathscr{M}_{\eta}(|T(f_1,f_2)|)(x)\\
&+\mathscr{M}(f_1)(x)\mathscr{M}(f_2)(x).
\end{align}
\end{small}
This together with Remark~\ref{comparison} proves \eqref{cotlar} . Now using the facts that $\mathscr{M}$ maps $L^{p,\infty}(X,d,\mu)$ to $L^{p,\infty}(X,d,\mu)$ for all $1<p<\infty$ and $T$ maps $L^1(X,d,\mu)\times L^1(X,d,\mu)$ to $L^{\frac{1}{2},\infty}(X,d,\mu)$, we get $T^*$ maps $L^1(X,d,\mu)\times L^1(X,d,\mu)$ to $L^{\frac{1}{2},\infty}(X,d,\mu)$. 
\qed

\bibliographystyle{plain}

\end{document}